\documentclass[11pt]{amsart}
\usepackage{amsbsy}
\usepackage{amssymb,amsthm,amsmath, latexsym}
\usepackage{mathrsfs}
\usepackage{color}
\usepackage[utf8]{inputenc}
\usepackage{graphicx,epsfig,subfigure,psfrag}

\usepackage{graphics}
\usepackage{enumerate}
 \usepackage[usenames,dvipsnames]{pstricks}

\begingroup
\theoremstyle{plain}
\newtheorem{theorem}{Theorem}[section]
\newtheorem{proposition}[theorem]{Proposition}
\newtheorem{lemma}[theorem]{Lemma}
\newtheorem{corollary}[theorem]{Corollary}

\theoremstyle{definition}
\newtheorem{definition}[theorem]{Definition}
\newtheorem{remark}[theorem]{Remark}

\newtheorem{conj}[theorem]{Conjecture}
\endgroup

\newcommand{\R}{\mathbb{R}}

\newcommand{\e}{\varepsilon}

\newcommand{\dx}{\mathrm{d} x}
\newcommand{\dy}{\mathrm{d} y}
\newcommand{\w}{w}
\date{\today}

\title{An isoperimetric problem with two distinct solutions}

\author{Antoine Henrot}
\author{Antoine Lemenant}
\author{Ilaria Lucardesi}

\begin{document}

\begin{abstract} In this paper we prove that among all convex domains of the plane with two axis of symmetry, the  maximizer of the first non trivial Neumann eigenvalue $\mu_1$ with perimeter constraint is achieved by the square and the equilateral triangle. Part of the result follows from a new general bound on $\mu_1$ involving the minimal width over the area. Our main result partially answers to a question addressed in 2009 by R. S. Laugesen,  I. Polterovich, and B. A. Siudeja.
\end{abstract}

\maketitle


\setcounter{tocdepth}{1}

\tableofcontents

\section{Introduction}

In this paper, we are interested in the  shape optimization of the second Neumann eigenvalue of the Laplacian, denoted by $\mu_1$, see Section \ref{section0} for the Notation. We underline that the shape optimization of the first Neumann eigenvalue $\mu_0$ is not interesting, since $\mu_0(\Omega)=0$ for every shape $\Omega$.

A well-known result, first proved by Szeg\H{o} (for plane simply-connected domains) and then generalized by Weinberger, says that the ball maximizes $\mu_1(\Omega)$ with volume constraint. The Dirichlet analogue is the Faber-Krahn inequality, which says that the ball minimizes the first Dirichlet eigenvalue of the Laplacian $\lambda_1(\Omega)$ under volume constraint.

But less is known for the perimeter constraint. Concerning the first Dirichlet  eigenvalue, it is easy to see that the ball is again the minimizer with fixed perimeter. Indeed, by successively using the isoperimetric inequality and then Faber-Krahn inequality, we get (here $\Omega$ is a bounded open set in $\R^N$ and $\omega_N$ is the volume 
of the unit ball $B_1$)
\begin{eqnarray}
P(\Omega)^{\frac{2}{N-1}}\lambda_1(\Omega) &\geq& \left(N\omega_N^{\frac{1}{N}}\right)^{\frac{2}{N-1}}  |\Omega|^{\frac{2}{N}} \lambda_1(\Omega) \notag \\
&\geq& 
\left(N\omega_N^{\frac{1}{N}}\right)^{\frac{2}{N-1}}  |B_1|^{\frac{2}{N}} \lambda_1(B_1)= P(B_1)^{\frac{2}{N-1}} \lambda_1(B_1),\notag
\end{eqnarray}
with equality for the ball in each of the above inequalities.

For the Neumann eigenvalue $\mu_1(\Omega)$, the maximizer  with perimeter constraint does not exist in general (see Proposition \ref{nonexistence1}). On the other hand, it is easy to see that a maximizer exists among convex sets (see Proposition \ref{existence}).  
However, up to our knowledge, even in dimension 2,  the convex maximizer is not known. The question was actually explicitly addressed in an Oberwolfach meeting in 2009 by R. Laugesen and I. Polterovich, \cite[Problem 9.2]{LS}, see also \cite[Open Problem 6.66]{LS2}, 
from which we can conjecture the following.

\begin{conj}\label{conj1} For all  planar convex domain   we have
\begin{eqnarray}
P^2(\Omega)\mu_1(\Omega)\leq 16\pi^2. \label{ineqConj}
\end{eqnarray}
The maximum being achieved by squares, and equilateral triangles.
\end{conj}

The eigenvalue $\mu_1(\Omega)$ cannot be computed explicitly in general. However, this can be done for simple domains, such as equilateral triangles, rectangles, and disks (see, e.g., \cite[Section 1.2.5]{livre_vert} and \cite{LS}). In particular, the equality in \eqref{ineqConj} for squares and equilateral triangles follows by a direct computation. Moreover, in \cite{LS} it was proved that among all triangles, the equilateral triangle is the maximizer. For the unit disk $\mathbb D$ (actually, for any disk, by scale invariance of the functional under study) we have  $P^2(\mathbb{D})\mu_1(\mathbb{D})=4\pi^2(j'_{1,1})^2 \simeq (4\times 3,39)\pi^2$ which is far from achieving the bound $16\pi^2$. This gap is large enough to imply that all regular polygons with a number of sides $N\geq 5$ will satisfy \eqref{ineqConj}, by comparison with a ball (see Proposition \ref{polygon}). Let us also mention that the conjecture is supported by numerical simulations, performed by Beniamin Bogosel (private communication). \\

In this paper we partially prove the conjecture. Our main result is the following.
\begin{theorem}\label{main} For all  planar convex domains with two axis of symmetry,
$$
P^2(\Omega)\mu_1(\Omega)\leq 16\pi^2.
$$
The maximum is achieved by squares and equilateral triangles and only by them.
 \end{theorem}

One of the key ingredients in the proof of Theorem \ref{main} is a new geometrical bound on $\mu_1(\Omega)$, valid for any bounded Lipschitz open set $\Omega \subset \R^2$ (see also Lemma \ref{boundmu1} for a general statement on any set, not necessarily Lipschitz). Besides the application to our problem, the result is interesting in itself.

\begin{lemma}\label{joli} For all planar bounded (Lipschitz) open set $\Omega$ we have
\begin{equation}\label{jolieq}
\mu_1(\Omega)\leq  \pi^2\frac{w(\Omega)^2}{|\Omega|^2}
\end{equation}
where $|\Omega|$ is the area of $\Omega$ and $w(\Omega)$ is the minimal width of $\Omega$. The inequality becomes an equality for all  and only rectangles.
\end{lemma}
The proof of this lemma is quite short and it is given in Section \ref{sectionI} (together with a generalization in higher dimension, see Lemma \ref{boundmu1} and
Lemma \ref{boundmu1N}). It allows us to conclude that Conjecture \ref{conj1} holds true for all planar convex domains with two orthogonal axis of symmetry. Actually, from Lemma \ref{joli} one can state the following immediate corollary.
\begin{corollary} \label{corr}  Any planar convex domain $\Omega$ satisfying 
\begin{eqnarray}
 P(\Omega) w(\Omega)\leq 4 |\Omega|, \label{inequG}
 \end{eqnarray}
also satisfies  inequality \eqref{ineqConj}.
\end{corollary}
As a matter of fact, any convex domain with two orthogonal axis of symmetry automatically satisfies \eqref{inequG} as it is proved in Proposition \ref{orthogonal}, which therefore answers to the conjecture for the case of convex domains with two orthogonal axis of symmetry, with a quite short proof. Moreover, among convex domains with two orthogonal axis of symmetry, only squares satisfy \eqref{jolieq} and \eqref{inequG}, and thus \eqref{ineqConj}, with an equality sign.
It is worth noticing that there exist other convex domains satisfying \eqref{inequG}, and thus Conjecture \ref{conj1}, without having necessarily two axis of symmetry: this is the case, e.g., of parallelograms (see Proposition~\ref{parallelogram}) and centrally symmetric convex domains (see Remark \ref{centrally}). Our study of parallelograms contains and generalizes the result by Raiko \cite{raiko}, with a completely different proof.

In order to obtain the main result (Theorem \ref{main}) in the generic case, we consider a domain $\Omega$ with two axis of symmetry with a given angle $\theta \in (0,2\pi)$. We rapidly arrive to the conclusion that it suffices to treat angles of the form $\theta=\pi/N$ with $N\geq 3$ (see Proposition \ref{axes}). Notice that when $\theta=\pi/4$, or more generically $\theta=\pi/N$ with $N=2^k$, $k\geq 2$, then the domain has two orthogonal axis of symmetry, and the argument above applies.

The case of $\theta=\pi/N$ with $N\geq 5$ is treated in Section \ref{sec-small}: for such domains, the comparison with the disk, which in general gives a too rough estimate, gives the desired bound with the strict inequality sign. 

The last case that we need to consider is that of an angle $\theta=\pi/3$, i.e., when $\Omega$ admits the same axis of symmetry as the equilateral triangle. This case is much 
more involved and it is treated in Section \ref{sectionIII}. We have to introduce a parameter $a\in [0, 1/3]$, which, roughly speaking, describes how far the domain is from an equilateral triangle.

The complicated situation occurs for small values of the parameter, $a\in [0,1/4]$, corresponding to shapes not far from the equilateral triangle. Here the strategy is to use a convex combination of the eigenfunctions of two suitable equilateral triangles as a test function for $\mu_1(\Omega)$.  After a series of hard computations and estimates, we come to the conclusion that nothing but a finite number of configurations have to be tested, which is finally done by computer.  The most technical parts are postponed to the Appendix.

On the other hand, for the remaining intermediate values of the parameter, namely $a\in [1/4, 1/3]$, we can easily conclude by using Szeg\H{o}-Weinberger's inequality together with a 
nice reverse isoperimetric inequality (involving the perimeter and the area). This concludes the proof of Theorem~\ref{main}.

\subsection*{Acknowledgements} The authors want to thank Beniamin Bogosel, Lorenzo Cavallina, Andrea Colesanti, Kei Funano, Richard Laugesen, and Shigeru Sakaguchi for the stimulating discussions. This work was partially supported by the project \emph{Shape Optimization (SHAPO)} n. ANR-18-CE40-0013 financed by the French \emph{Agence Nationale de la Recherche (ANR)}, and by the Lorraine and Tohoku Universities through a joint research project. IL is member of the Italian research group \emph{Gruppo Nazionale per l'Analisi Matematica, la Probabilit\`{a} e le loro Applicazioni (GNAMPA)} of the \emph{Istituto Nazionale di Alta Matematica (INdAM)}.

\section{Notation and useful bounds}\label{section0}

We start this section by fixing the notation that will be used throughout the paper. Then we give some known bounds of $P(\Omega)$ and $\mu_1(\Omega)$  coming from the literature that will be used later. As a direct application of one of these bounds (Szeg\H{o}-Weinberger's inequality) we give a short proof of the validity of Conjecture \ref{conj1} for regular polygons.

\subsection*{Notation.} 

We denote by $Vect[\xi_1, \ldots, \xi_{m}]$, $m\in \mathbb N$, the vector space of $\mathbb R^N$ generated by the vectors $\xi_1, \ldots, \xi_m \in \mathbb R^N$. The orthogonal complement of a subspace $V$ of  $\mathbb R^N$ is denoted by $V^\perp$.

Given three points $A,B,C$ in the plane, we denote by $\overline{AB}$ the length of the segment joining $A$ and $B$, by $(AB)$ the line passing through $A$ and $B$ (when they do not coincide), and by $\widehat{ABC}$ the angle of vectors between $\overset{\longrightarrow}{AB}$ and $\overset{\longrightarrow}{BC}$.

Given a set $\Omega$, we denote its area by $|\Omega|$, its perimeter (always defined when $\Omega$ is convex) by $P(\Omega)$, its diameter by $D(\Omega)$, and its convex hull by $\mathrm{conv}(\Omega)$. We denote by $\mathcal{H}^1$ the one dimensional Hausdorff measure,  $\mathbb{S}^1$ the unit circle, and $\mathbb{D}$ the unit disk. The minimal width of the convex set $\Omega \subset \R^2$ (or more generally of any bounded open set)
is defined by 
$$
w(\Omega):=\min_{\nu \in \mathbb{S}^1} \mathcal{H}^1(p_\nu(\Omega)),$$
where $p_\nu : \R^2 \to \R \nu$ is the orthogonal projection onto the vectorial line $\R \nu$ oriented by the vector $\nu$. 

If $\Omega$ is a bounded Lipschitz domain (as for instance a convex domain), the spectrum of the Neumann Laplacian is discrete, and the first Neumann eigenvalue is always zero. We will denote by $\mu_1(\Omega)$ the second eigenvalue, defined by 
\begin{equation}\label{mu1}
\mu_1(\Omega)=\min_{u \in H^1(\Omega) \; : \; \int_{\Omega} u \;\dx=0} \frac{\int_{\Omega} |\nabla u|^2 \; \dx}{\int_{\Omega} u^2 \;\dx}.
\end{equation}

\subsection*{Some useful bounds.}
The first inequality that we mention is the following: for all  convex domains (with non empty interior) 
\begin{eqnarray}
P(\Omega)> 2D(\Omega). \label{inclusion}
\end{eqnarray}

The second inequality says that among all domains the ball maximizes $\mu_1$ with fixed volume. More precisely (see \cite{W56} or  \cite[7.1.2]{livre_vert}), for any Lipschitz domain $\Omega \subset \R^2$ we have
\begin{eqnarray}
|\Omega|\mu_1(\Omega)\leq \pi \mu_1(\mathbb{D})=\pi(j'_{1,1})^2, \quad\quad \text{ (Szeg\H{o}-Weinberger)} \label{SzegoW}
\end{eqnarray}
where $j'_{1,1}$ is the first zero of the derivative of the Bessel function $J_1$. In particular, $(j'_{1,1})^2\simeq 3.39$ (see \cite[page 11]{livre_vert}).

The third inequality is a classical lower bound for $\mu_1$, valid for any convex domain $\Omega\subset \R^2$ (see \cite{PW}),
\begin{eqnarray}
D(\Omega)^2 \mu_1(\Omega)\geq \pi^2. \quad\quad \text{ (Payne-Weinberger)} \label{PW}
\end{eqnarray}

The last inequality  that we mention gives a counterpart of the above inequality, with an upper bound. It was explicitly stated by Ba\~{n}uelos and Burdzy in \cite[Corollary 2.1]{BB}, but actually it follows from a more general result by Cheng \cite{cheng}: for any convex domain $\Omega \subset \R^2$,  
 
 \begin{eqnarray}
D(\Omega)^2 \mu_1(\Omega)\leq 4j_{0,1}^2. \quad\quad \text{ (Cheng, Ba\~{n}uelos-Burdzy)} \label{cheng}
\end{eqnarray}
Here $j_{0,1}$ is the first zero of the Bessel function $J_0$, and it holds $j_{0,1}\simeq 2.405$ (see \cite[page 11]{livre_vert}).

\subsection*{The case of regular polygons}\label{polygons}
As any regular polygon admits at least two axis of symmetry, we can treat them as a direct consequence of our main result (Theorem~\ref{main}). However, Szeg\H{o}-Weinberger's inequality allows us to give an independent and shorter proof.
\begin{proposition} \label{polygon}For any regular polygon $\Omega \subset \R^2$ inequality \eqref{ineqConj} holds true. 
\end{proposition}

\begin{proof} Let $\Omega$ be a regular polygon with $N$ sides. We denote by $A$ its area and $P$ its perimeter. Then a simple computation shows that
$$
P=2N\sin\left(\frac{\pi}{N}\right), \quad A=N\sin\left(\frac{\pi}{N}\right)\cos\left(\frac{\pi}{N}\right).
$$
Next, by Szeg\H{o}-Weinberger's inequality \eqref{SzegoW}, we know that 
$$
A\mu_1(\Omega)\leq \pi \mu_1(\mathbb{D}).
$$
It follows that  \eqref{ineqConj} will be true if $P^2\mu_1(\mathbb{D})\leq 16 A \pi$ or, equivalently, if
$$
\mu_1(\mathbb{D})4 N^2 \sin^2\left(\frac{\pi}{N}\right) \leq 16N\sin\left(\frac{\pi}{N}\right)\cos\left(\frac{\pi}{N}\right)\pi.
$$
Since  $\cos\left(\frac{\pi}{N}\right)$ never vanishes for $N\geq 3$, we can transform the relation into
 $$
 N \tan\left(\frac{\pi}{N}\right) \leq \frac{4\pi}{\mu_1(\mathbb{D})} .
$$
Now $\mu_1(\mathbb{D})\simeq 3.39$ thus $\frac{4\pi}{\mu_1(\mathbb{D})}>3.706$. On the other hand, the function $x\mapsto  x \tan\left(\frac{\pi}{x}\right)$ is decreasing on $[5,+\infty[$ and 
$$
5 \tan\left(\frac{\pi}{5}\right)<3.633.
$$
This means that \eqref{ineqConj} is satisfied for every polygon with $N\geq 5$. For $N=3,4$, as we already know, the desired inequality is an equality.
\end{proof}

\section{Existence and non existence results}\label{sectionexis}

In this section we prove the existence of a convex maximizer for the quantity $P^2(\Omega)\mu_1(\Omega)$. Then we stress the non existence for the analogue (again among convex sets) minimizing problem and for the maximizing problem without convexity constraint.

\begin{proposition}\label{existence} The problem
$$\max\{P^2(\Omega)\mu_1(\Omega) \quad : \quad \Omega \subset \R^2 \text{, convex }\},$$
admits a solution.
\end{proposition}
\begin{proof}Let $\{\Omega_n\}_{n\in \mathbb N}$ be a maximizing sequence (of plane convex domains). By the scale invariance, we can assume that all these domains have diameter equal to 1, so that they are all contained into a fixed ball. By Blaschke selection theorem, only two situations can occur:
\begin{enumerate}
\item either there exists a convex  set $\Omega^*$ with non empty interior such that $\Omega_n$ (or a subsequence) converges to $\Omega^*$  for the  complementary Hausdorff distance, as $n\to \infty$. In that case, as the quantities $P(\Omega_n)$ and $\mu_1(\Omega_n)$  are continuous for the Hausdorff convergence in the class of convex domains, the set $\Omega^*$ will be the maximizer (for the continuity of $P(\Omega_n)$ see \cite[Lemma 2.3.5]{BcB} and for the continuity of $\mu_1(\Omega_n)$ see \cite{Henrot-Pierre} or \cite{henrot-ftouhi}); 
\item or the sequence $\Omega_n$ converges (for the complementary Hausdorff distance) to a segment of diameter 1, as $n\to \infty$. But in that case, the perimeter of $\Omega_n$ converges to $2 = 2D(\Omega_n)$, thus 
$$\lim_n P^2(\Omega_n )\mu_1 (\Omega_n ) = 4\lim_n D^2(\Omega_n )\mu_1(\Omega_n ) \leq 16 (j_{0,1})^2$$
by Cheng inequality \eqref{cheng}, thus we would have  
$${\rm sup}P^2(\Omega)\mu_1(\Omega) \leq  16 (j_{0,1})^2 < 16\pi^2$$
 that is impossible.
\end{enumerate}
The proof is therefore completed.
\end{proof}

Regarding to the minimization of $P^2(\Omega) \mu_1(\Omega)$, we notice that a minimizer does not exist as stated in the following proposition.

\begin{proposition} For all planar convex domains  we have 
\begin{eqnarray}
P^2(\Omega) \mu_1(\Omega) {>} 4 \pi^2. \label{lowerB}
\end{eqnarray}
 The lower bound is optimal and it is reached asymptotically  by a sequence of rectangles shrinking to a segment.
\end{proposition}
\begin{proof}  Payne-Weinberger's  inequality  \eqref{PW} says that 
$$D^2(\Omega)\mu_1(\Omega)\geq \pi^2,$$
where $D(\Omega)$ is the diameter of $\Omega$. On the other hand we can use \eqref{inclusion} which, for the recall, says that for any planar convex domain we always have $P(\Omega)> 2 D(\Omega)$. These two inequalities give \eqref{lowerB}. Now by considering a sequence of rectangles  $\Omega_\varepsilon:=[0,1]\times [0,\varepsilon]$, $0<\e <<1$,
that shrink to the segment $[0,1]\times\{0\}$ as $\e \to 0$, we observe that  
$$P^2(\Omega_\varepsilon)\mu_1(\Omega_\varepsilon)=4(1+\varepsilon)^2 \pi^2\to 4\pi^2.$$
Here we have used the exact value of $\mu_1$ for rectangles of the form $[0,L]\times [0,l]$ with $L\geq l$, which reads $\mu_1([0,L]\times[0,l])= \frac{\pi^2}{L^2}$ (see, e.g., \cite[Proposition 1.2.13]{livre_vert}). 
\end{proof}

We now show that if we relax the convexity constraint, then the maximizing problem has no solution.

\begin{proposition} \label{nonexistence1}The following equality holds:
$$\sup\left\{P^2(\Omega)\mu_1(\Omega) \quad : \quad \Omega \subset \R^2 \right\}=+\infty.$$
\end{proposition}
\begin{proof} 

Let $a>0$ be a small parameter (that will be taken infinitesimal at the end of the proof) and let $Q_a=[0,a]\times [0,a]$ be a square of size $a$.  We will construct a sequence of sets $\Omega_n$, $n\in \mathbb N$, such that for every $n\in \mathbb N$ the perimeter satisfies $P(\Omega_n)=1$, whereas, in the limit as $n\to \infty$, $\Omega_n \to Q_a$ with respect to the Hausdorff distance. For that purpose, for a given $n$ we construct a piecewise linear affine function $f_n:[0,a]\to \R$ as follows. We prescribe its value on a uniform partition of $[0,a]$ and then we extend it in an affine way:
{\renewcommand{\arraystretch}{1.5}
$$
\left\{
\begin{array}{lc}
f_n(a\frac{k}{2n})=0 & \text{ for all } k=0,\dots,n-1,   \\
f_n(a(\frac{k}{n}+\frac{1}{2n})) = b \frac{a}{2n}  & \text{ for all } k=0,\dots,n-1,   \\ 
\text{ linear affine otherwise,}
\end{array}
\right.
$$
}
where the parameter $b$ is
$$b:=\sqrt{\frac{1}{16 a^2}-1}.$$

The graph of $f_n$ is the union of $n$ isosceles triangles with base of length $a/n$ and legs of length
$$\sqrt{\left(\frac{a}{2n}\right)^2+\left(\frac{ba}{2n}\right)^2}=\frac{1}{8n}.$$
In particular, the graph of $f_n$ over $[0,a]$ has length exactly $1/4$, and does not depend on $n$. Moreover, since $\|f_n\|_\infty =  ab/(2n)\to 0$, it follows that $f_n \to 0$ uniformly in $[0,a]$, as $n\to \infty$.  Finally, we notice that $f_n$ is uniformly Lipschitz with constant $b$. Then we reproduce the graph of $f_n$ over each side of the square $Q_a$, see Figure \ref{Figure1} below: this procedure allows to construct (the boundary of) a sequence of domains $\Omega_n$, uniformly Lipschitz, converging to $Q_a$ with respect to the Hausdorff distance.

\begin{figure}[h!]                                             
\begin{center}
\includegraphics[height=7truecm]{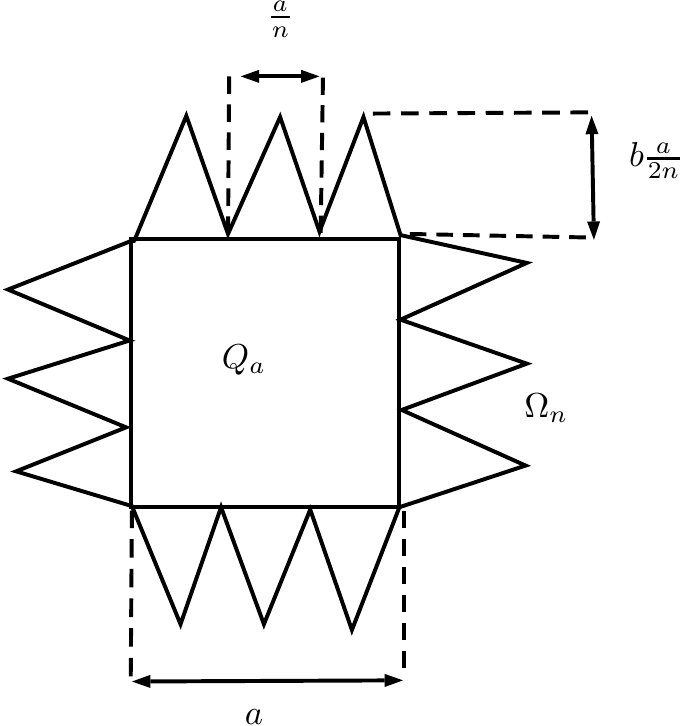}
\caption{\textit{The construction of the sequence of sets $\Omega_n$.}}\label{Figure1}
\end{center}
\end{figure}

By applying  \cite{henrot-ftouhi}, (see also \cite[Section 2.3]{livre_vert}) we know that $\mu_1(\Omega_n) \to \mu_1(Q_a)=\frac{\pi^2}{a^2}$.  On the other hand $P(\Omega_n)=1$ for all $n$, so that $P^2(\Omega_n)\mu_1(\Omega_n)\to \frac{\pi}{a^2}$. Consequently,
$$\sup\left\{P^2(\Omega)\mu_1(\Omega) \quad : \quad \Omega \subset \R^2 \right\} \geq \frac{\pi^2}{a^2},$$
and we conclude by finally letting   $a\to 0$ in the above inequality. 
\end{proof}
\begin{remark}
Let us observe that there exist domains with infinite perimeter but with a finite Neumann eigenvalue. This is, in particular, the case of some
fractal domains as the Koch snowflake. Obviously, this provides another proof of Proposition \ref{nonexistence1}.
\end{remark}

\section{A general bound on $\mu_1$ and the case of orthogonal axis of symmetry}\label{sectionI}

This section is devoted to a new upper bound of $\mu_1$, which is valid for generic planar domains (not necessarily convex) and which involves two geometric functionals: the area and the minimal width. In the convex setting, this allows us to show that the conjecture holds true for a broad class of domains: those for which the perimeter satisfies a suitable upper bound that matches the new upper bound of $\mu_1$. In the second part of the section, we provide some examples of such domains, proving in particular 
the conjecture \eqref{ineqConj} for convex sets with two orthogonal axis of symmetry.

\begin{lemma}\label{boundmu1} For every (non-empty) open and bounded set $\Omega\subset \R^2$,  we have 
\begin{equation}\label{inf}
\inf_{u \in H^1(\Omega) \; : \; \int_{\Omega} u \dx=0} \frac{\int_{\Omega} |\nabla u|^2 \; \dx}{\int_{\Omega} u^2 \;\dx} \leq \pi^2\frac{w^2(\Omega)}{|\Omega|^2}.
\end{equation}
In particular, if $\mu_1(\Omega)$  is well defined, then $\mu_1(\Omega)\leq \pi^2\w^2(\Omega)/|\Omega|^2$.\\
Equality occurs in this inequality only for rectangles.
\end{lemma}

\begin{proof} Let $\Omega\subset \mathbb R^2$ be a (non-empty) bounded open set. Throughout the proof, let $A:=|\Omega|$ and $\w:=\w(\Omega)$.
Without loss of generality, we may assume that the coordinate system is chosen so that the minimal width of $\Omega$ is achieved in the vertical direction. In other words $\Omega$ is contained into the horizontal strip $S:= \R \times [-w/2 , w/2]$. In $S$ we consider the rectangle 
$$
R:=[-L/2 , L/2]\times [-\w/2 , \w/2], \quad \hbox{with }L:=A/\w,
$$
having area $|R|=A$, and we define the function
$$
u(x,y):=
\left\{
\begin{array}{lll}
\sin(\pi x /L) & \text{ in } R\\
+1 & \text{ in } S \setminus R, \ x\geq L/2 \\
-1 & \text{ in } S \setminus R,\  x\leq -L/2.
\end{array}
\right.
$$
Since $\Omega\subset \subset S$, the function $u$ is well defined in $\Omega$, but it does not have (in general) zero average in it. For $t\in \mathbb R$, let us denote by $\Omega_{t}$ the translation $\Omega_{t}:=\Omega + (t,0)$, which is still contained into the strip $S$. The map $t \mapsto \int_{\Omega_{t}} u $ is continuous, moreover, by construction of $u$ it is negative for $t<0$ large enough in modulus, and it is positive for $t>0$ large enough. Therefore, there exists $x_0\in \mathbb R$ such that $\int_{\Omega_{x_0}} u= 0 $. Since all the functionals appearing in \eqref{inf} are invariant under translation, it is enough to prove the statement for $\Omega_{x_0}$.
The choice of $x_0$ allows us to obtain
$$
\inf_{f \in H^1(\Omega_{x_0}) \; : \; \int_{\Omega_{x_0}} f \dx=0} \frac{\int_{\Omega_{x_0}} |\nabla f|^2 \;\dx}{\int_{\Omega_{x_0}} f^2 \;\dx}   \leq   \frac{\int_{\Omega_{x_0}} |\nabla u|^2 \;\dx}{\int_{\Omega_{x_0}} u^2 \;\dx}.
$$
Let us bound from above the numerator and from below the denominator of the Rayleigh quotient:
\begin{eqnarray}
\int_{\Omega_{x_0}} |\nabla u|^2 \;\dx & =& \int_{\Omega_{x_0}\cap R} |\nabla u|^2 \;\dx \leq  \int_R |\nabla u|^2\;\dx;  \notag \\
\int_{\Omega_{x_0}}  u^2 \;\dx & =& \int_{\Omega_{x_0}\cap R} u^2 \;\dx + |\Omega_{x_0} \setminus R| \notag \\
&= & \int_R u^2 \;\dx   - \int_{R\setminus \Omega_{x_0}} u^2 \;\dx +  |\Omega_{x_0} \setminus R| \notag \\
&\geq & \int_{R} u^2 \;\dx - |R\setminus \Omega_{x_0}| +  |\Omega_{x_0} \setminus R|  =  \int_{R} u^2 \;\dx. \notag 
\end{eqnarray}

\begin{figure}[h]                                             
\begin{center}
\includegraphics[height=7truecm]{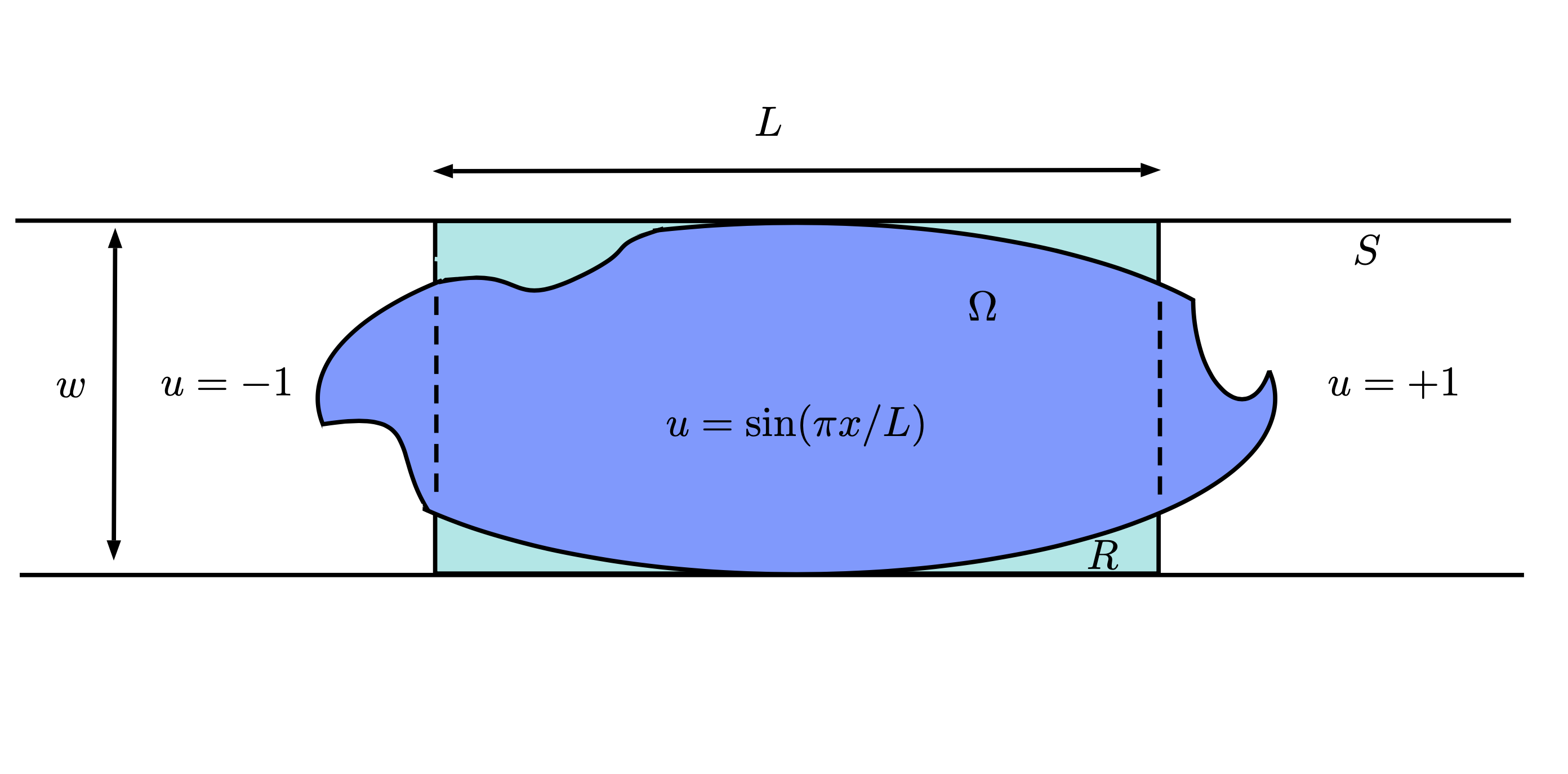}
\caption{\textit{The competitor $u$ for the domain $\Omega$.}}\label{Figure2}
\end{center}
\end{figure}

In the lower bound of the denominator we have used $|u|\leq 1$ (true by construction) and $|\Omega_{x_0} \setminus R|= |R\setminus \Omega_{x_0}|$ (true since $|\Omega_{x_0}|=A=|R|$). A direct computation gives
$$
 \int_R |\nabla u|^2\;\dx = \frac{\pi^2 \w}{2L}, \quad  \int_{R} u^2 \;\dx = \frac{\w L}{2},
$$
so that, recalling the choice $L=A/\w$, 
$$
\inf_{f \in H^1(\Omega_{x_0}) \; : \; \int_{\Omega_{x_0}} f \dx=0} \frac{\int_{\Omega} |\nabla f|^2 \;\dx}{\int_{\Omega_{x_0}} f^2 \;\dx}  \leq \frac{\pi^2}{L^2}= \frac{\pi^2\w^2}{A^2},
$$
which concludes the proof of \eqref{inf}. The equality case comes immediately from the study of the two chains of inequalities. \end{proof}

 In the present paper we deal with planar shapes, however Lemma \ref{boundmu1} can be generalized in higher dimension, as we show in Lemma \ref{boundmu1N}. To state this result, we need some notation.
Let $\Omega$ be any bounded set in $\mathbb{R}^N$. We successively define, for $k=1,2,\ldots N-1$:
\begin{itemize}
\item $w_1$ the minimal width of $\Omega$ and $\xi_1$ the unit vector defining the direction of this minimal width,
\item $w_2$ the minimal width of the orthogonal projection of $\Omega$ on the space $Vect[\xi_1]^\perp$ and $\xi_2$ the unit 
vector defining the direction of this minimal width,
\item $\vdots$
\item $w_k$ the minimal width of the orthogonal projection of $\Omega$ on the vector space $Vect[\xi_1,\xi_2,\ldots,\xi_{k-1}]^\perp$  and $\xi_k$ the unit vector defining the direction of this minimal width.
\end{itemize}
Then we can state:
\begin{lemma}\label{boundmu1N} For every open, non-empty and bounded set $\Omega\subset \R^N$, we have 
\begin{equation}\label{inf2b}
\inf_{u \in H^1(\Omega) \; : \; \int_{\Omega} u \; \dx=0} \frac{\int_{\Omega} |\nabla u|^2 \; \dx}{\int_{\Omega} u^2 \;\dx} \leq
 \pi^2\frac{\prod_{k=1}^{N-1}w_k^2}{|\Omega|^2}.
\end{equation}
In particular, if $\mu_1(\Omega)$  is well defined, then $\mu_1(\Omega)\leq \pi^2\prod_{k=1}^{N-1}w_k^2/|\Omega|^2$.
\end{lemma}
The proof of this Lemma follows exactly the same line as in the two-dimensional case by introducing the cuboid 
$$K=\left[-\frac{L}{2}, \frac{L}{2}\right]\times \left[-\frac{w_{N-1}}{2},\frac{w_{N-1}}{2}\right]\times \ldots \times \left[-\frac{w_1}{2}, \frac{w_1}{2}\right]$$ 
with $L$ chosen such that $L \prod_{k=1}^{N-1} w_k = |\Omega|$ and $u$ the same test function depending only on the first variable $x_1$.
The proof is left to the reader.
 
 \medskip
A direct consequence of Lemma \ref{boundmu1} is the immediate Corollary \ref{corr} stated in the Introduction, which says that the convex sets for which 
\begin{eqnarray}
 P(\Omega) w(\Omega)\leq 4 |\Omega| \label{inequG2}
 \end{eqnarray}
satisfy the conjecture \eqref{ineqConj}. In what follows we provide some examples. We start with parallelograms: the next proposition generalizes the result contained in \cite{raiko}, where a similar statement has been proved only for some parallelograms of particular type. 

\begin{proposition}\label{parallelogram} All parallelograms satisfy \eqref{inequG2}. In particular all parallelograms satisfy \eqref{ineqConj}.  
\end{proposition}

\begin{proof} Let $\Omega$ be a parallelogram with area $A$, minimal width $\w$, and perimeter $P$. Let $L$ be largest side and $\ell$ the smaller side i.e. $\ell\leq L$. Let $\theta$ be the smallest angle, as in Figure \ref{Figure3}. 
 \begin{figure}[h]                                             
\begin{center}
\includegraphics[height=4truecm]{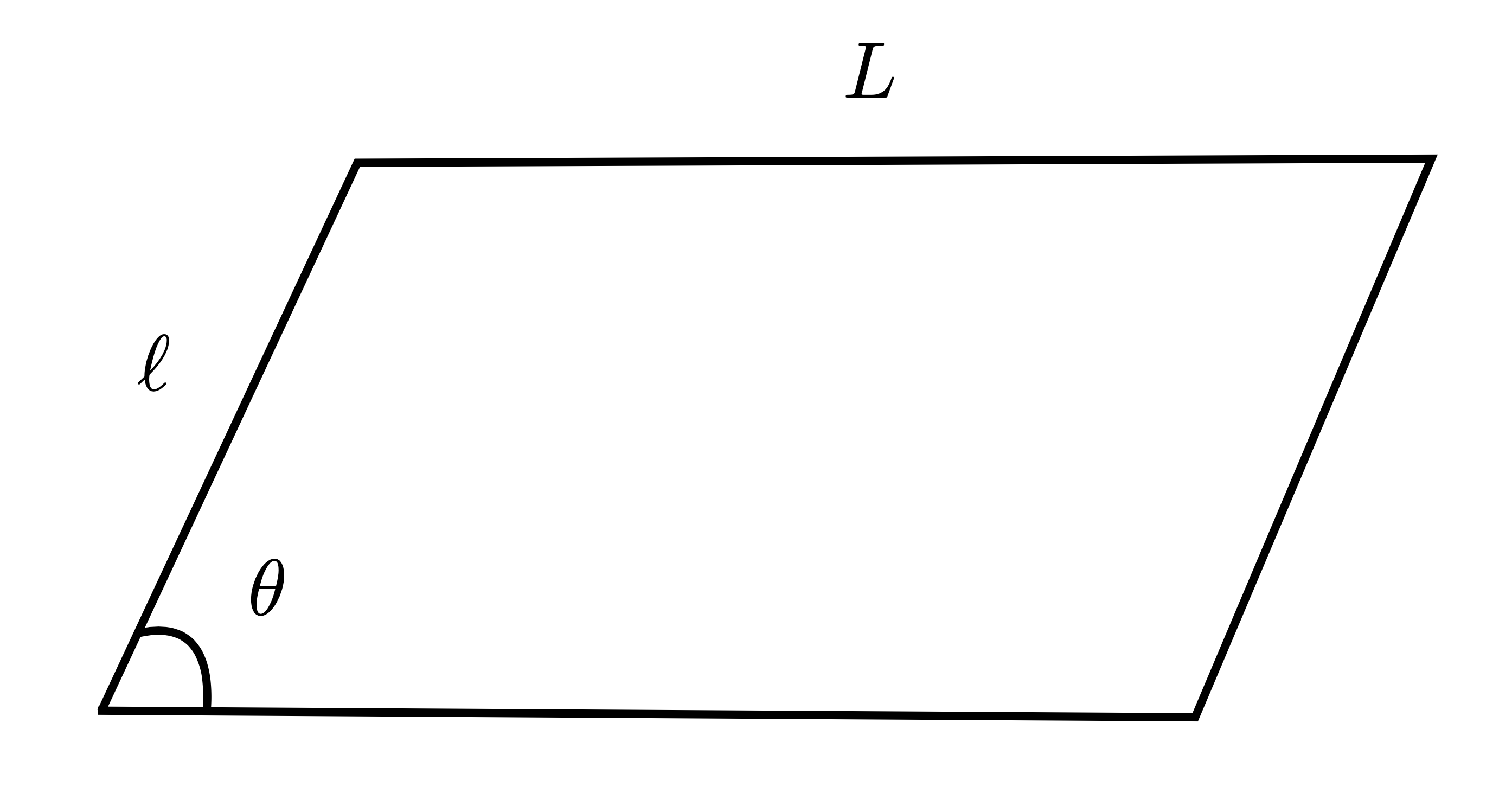}\\
\caption{\textit{A parallelogram.}}\label{Figure3}
\end{center}
\end{figure}
It is easy to see that
$$
P=2\ell +2 L, \quad w \leq   \ell \sin(\theta), \quad  A=\sin(\theta)\ell L.
$$ 
These inequalities imply \eqref{inequG2}:
$$
P\w\leq (2\ell +2L) \sin(\theta)\ell\leq 4 L \sin(\theta)\ell  = 4A.
$$
The validity of \eqref{ineqConj} follows by Corollary \ref{corr}.
\end{proof}

Another family of domains satisfying $P(\Omega)\w(\Omega) \leq 4 |\Omega|$ is that of convex sets with two orthogonal axis of symmetry. 
\begin{proposition}\label{orthogonal} All planar convex domains with two orthogonal axis of symmetry satisfy \eqref{inequG2}. In particular all planar convex domains with two orthogonal axis of symmetry satisfy inequality \eqref{ineqConj}; moreover, the equality sign is satisfied only for squares.
\end{proposition}

\begin{proof} In virtue of Corollary \ref{corr}, to get \eqref{ineqConj} we only need to prove \eqref{inequG2}. By approximation of $\Omega$ with polygons, we may directly assume that $\Omega$ is a polygon. We also assume that the two axis of symmetry are  $\{x=0\}$ and $\{y=0\}$. We consider the part of $\Omega$ in the first quadrant $\{x\geq 0, \ y\geq 0\}$: by connecting the origin with the vertexes of the polygon, we divide $\Omega$ into triangles. Let $N$ be the number of such triangles (in the first quadrant). For $i=1, \ldots, N$, the $i$-th triangle has one side on $\partial \Omega$, whose length is denoted by $a_i$. The corresponding
 height is an orthogonal segment passing through the origin, whose length is denoted by $h_i$. According to this notation, we have
$$
P(\Omega) = 4 \sum_{i=1}^N a_i, \quad |\Omega|= 2 \sum_{i=1}^N a_i h_i.
$$
The key point of the proof relies on the following elementary observation about the minimal width (see Fig. \ref{Figure4}):
\begin{equation}\label{cs}
\forall i , \quad \quad 2 h_i\geq \w(\Omega). 
\end{equation}
 \begin{figure}[h]                                             
\begin{center}
\includegraphics[height=6truecm]{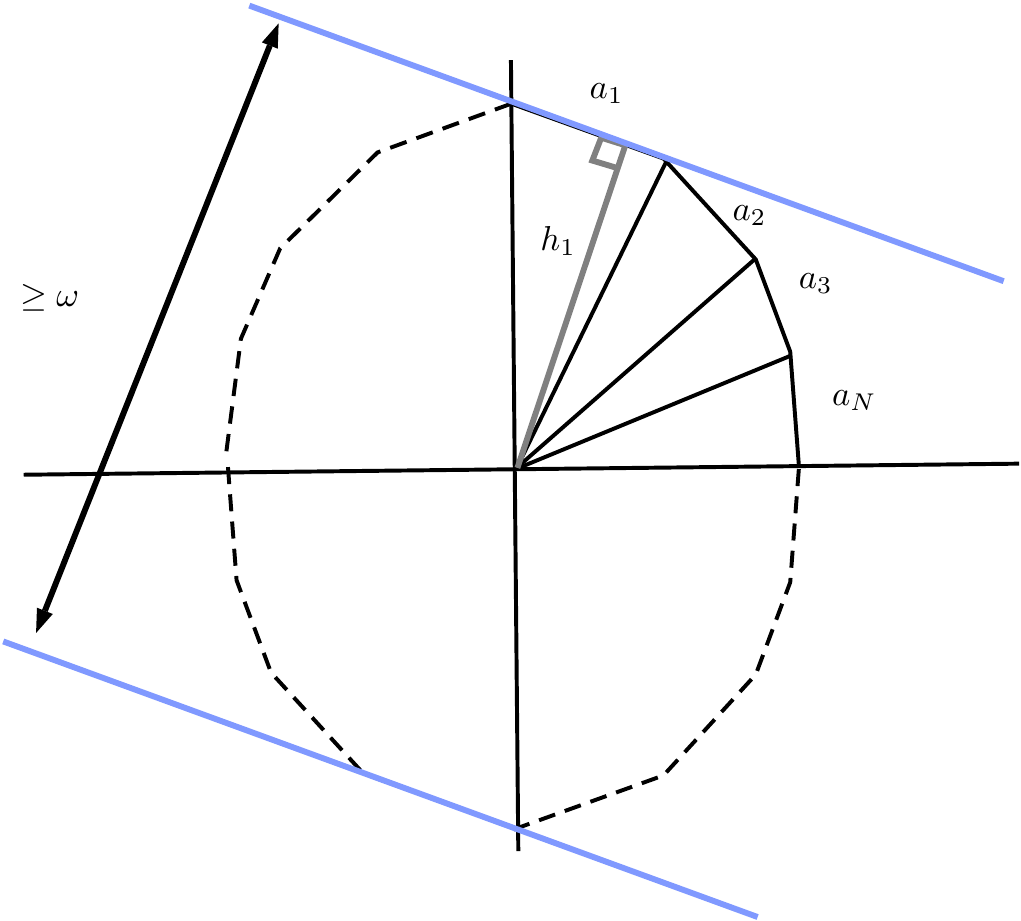}
\caption{\textit{Polygons with two orthogonal axis of symmetry.}}\label{Figure4}
\end{center}
\end{figure}
Thus we infer that
$$
|\Omega| \geq \sum_{i=1}^N a_i \w(\Omega) =\frac14  \w(\Omega) P(\Omega),
$$
that is, \eqref{inequG2}. This concludes the proof of \eqref{ineqConj}. Let us consider the equality cases: take a planar convex domain $\Omega$ with two orthogonal axis of symmetry, such that $P^2(\Omega) \mu_1(\Omega)=16 \pi^2$. Then, in view of \eqref{inf} and \eqref{ineqConj}, it must satisfy
$$
\mu_1(\Omega) = \pi^2 \frac{\w^2(\Omega)}{|\Omega|}, \quad P(\Omega)\w(\Omega) = 4 |\Omega|.
$$
The former, thanks to Lemma \ref{boundmu1}, implies that $\Omega$ is a rectangle, say $\Omega=[0,L]\times [0,\ell]$, $\ell\leq L$. For rectangles, the second equality is satisfied if and only if
$$
2(L+\ell) \ell = 4 L \ell,
$$
namely if and only if $L=\ell$ and $\Omega$ is a square. This concludes the proof.\end{proof}

\begin{remark}\label{centrally}The key point in the proof of Proposition \ref{orthogonal} is the inequality \eqref{cs}, which relates the distance of two supporting lines to the height of the triangles. The inequality holds true more in general for centrally symmetric convex sets, implying the validity of the conjecture \eqref{ineqConj} for these domains. This fact has been pointed out by K. Funano and A. Colesanti, while discussing on the problem.
\end{remark}

\section{The case of axis of symmetry forming a small angle}\label{sec-small}

This section contains two results: first we show that  we can restrict ourselves to the case in which the angle between two axis of symmetry is of the form $\pi/N$, $N\geq 3$; then we analyse the case of small angles of the form $\pi/N$ with $N\geq 5$, proving the validity of \eqref{ineqConj} in this framework. Note that the case $\pi/4$ has been treated in the previous section, while the case $\pi/3$ will be the object of the next section.

\begin{proposition}\label{axes}Let $\Omega\subset \R^2$ be any planar convex domain which admits two axis of symmetry with respective angle $\theta\in (0,2\pi)$. Then:
\begin{enumerate}
\item either $\theta = \alpha 2\pi$ with $\alpha \not \in \mathbb{Q}$ and $\Omega$ is a disk;
\item or $\theta = \alpha 2\pi$ with $\alpha \in \mathbb{Q}$ and $\Omega$ admits two (possibly other) axis of symmetry, with respective angle  $\frac{\pi}{N}$ with $N\geq 3$.
\end{enumerate}
\end{proposition}
\begin{proof} The proof is divided into three steps. In the following, $\Omega$ is a planar convex set with two axis of symmetry $L_1$ and $L_2$, forming and angle $\theta\in (0,2\pi)$.

\emph{Step 1: invariance under rotation.} 
Since the composition of two symmetries produces the rotation of twice the angle, we infer the following two properties of $\Omega$:
\begin{itemize}
\item[i)] it is invariant under rotation of angle $2\theta$, denoted by $R_{2\theta}$;
\item[ii)] if $L$ is an axis of symmetry, then $R_{2\theta}(L)$ is an axis of symmetry, too.
\end{itemize}

By repeatedly applying (ii) to $L_1$ and $L_2$, we infer that $R_{2\theta}^k(L_1)$ and $R_{2\theta}^k(L_2)$, $k\in \mathbb N^*$, are axis of symmetry, forming with $L_1$ an angle of $2k\theta$ and $\theta + 2k\theta$, respectively. Therefore, for every $m\in \mathbb N^*$, there exist an axis of symmetry (either a rotation of $L_1$ or a rotation of $L_2$) forming an angle of $m\theta$ with $L_1$. In view of (i) we infer the set $\Omega$ is invariant under all the rotations $R_{2m \theta}$, $m\in \mathbb N^*$.

\emph{Step 2: case of $\theta$ non rational multiple of $2\pi$.} Assume that $\theta=\alpha2\pi$, for some $\alpha \in \mathbb R \setminus \mathbb Q$. Then the set 
$$
\{k\alpha2\pi + 2j\pi \quad : \quad k\in \mathbb{Z} \text{ and }  j \in \mathbb{Z}\}
$$
is dense in $\R$ (because a subgroup of $\R$ is either discrete or dense, and this one can not be discrete). In this case we deduce that $\Omega$ is invariant under a dense set of rotations, therefore $\Omega$ is a disk, which proves (1).

\emph{Step 3: case of $\theta$ rational multiple of $2\pi$.} Let now $\theta = \alpha 2 \pi$ with $\alpha \in \mathbb Q$, say $\alpha=\frac{p}{q}$ with $p$ and $q$ in $\mathbb Z$. Thanks to Bezout Lemma, we can find two integers $a,b\in \mathbb{Z}$ such that 
$$ap+bq=1.$$
Then we can write
$$a\theta=a \alpha 2 \pi = a\frac{p}{q}2\pi=\frac{1}{q}2\pi-b2\pi,$$
namely there exists a multiple of $\theta$ which is equal to $\frac{2\pi}{q}$ mod $2\pi$.
By Step~1, this implies that there exist two axis of symmetry with respective angle $\tilde{\theta}:=2\pi/q$, $q\in \mathbb N$. Let $L'$ and $L''$ denote the two axis. Without loss of generality, we may take $L'$ horizontal. If $q$ is even, (2) is proved with $N=q/2$. Let us assume that $q$ is odd, of the form $q=2h+1$, for some $h\in \mathbb N^*$. If $h$ is even, we apply (ii) $h/2$ times to the axis $L'$: the rotated axis $R_{2\tilde{\theta}}^{h/2}(L')$ is again an axis of symmetry, forming with $L'$ two supplementary angles:
$$
\frac{h}{2} (2 \tilde{\theta})= \frac{2 h}{2h+1} \pi \quad \hbox{and} \quad \pi -  \frac{h}{2} (2 \tilde{\theta}) = \frac{\pi}{2h+1} = \frac{\pi}{q}.
$$
The second angle concludes the proof of (2) with $N=q$.
The remaining case is when $\tilde{\theta}=2\pi/q$ with $q=2h+1$ and $h$ odd. Here we apply (ii) $(h-1)/2$ times to $L''$: the rotated axis $R_{2\tilde{\theta}}^{(h-1)/2}(L'')$ is again an axis of symmetry, forming with $L'$ two supplementary angles:
$$
\tilde{\theta} + \frac{h-1}{2}(2\tilde{\theta}) = \frac{2h}{2h+1} \pi  \quad \hbox{and} \quad \pi - \left[\tilde{\theta} + \frac{h-1}{2}(2\tilde{\theta}) \right]= \frac{\pi}{q}, 
$$
giving (2) with $N=q$. This concludes the proof.
\end{proof}

Let us now prove the main result when the angle between the two axis of symmetry is small. In view of the previous proposition, this covers almost all the possible cases.

\begin{proposition}\label{sym5} Let $\Omega$ be a planar convex domain with two axis of symmetry with angle $\theta= \frac{\pi}{N}$ with $N\geq 5$. Then the inequality \eqref{ineqConj} holds true strictly. 
\end{proposition}

\begin{proof} The strategy of the proof is to compare $\Omega$ with a disk by using the Szeg\H{o}-Weinberger's inequality  \eqref{SzegoW}:
\begin{eqnarray}
\mu_1(\Omega)\leq \frac{\pi(j'_{1,1})^2}{|\Omega|}. \label{SSe}
\end{eqnarray}
In virtue of \eqref{SSe}, the proposition will be proved if we are able to establish 
$$\frac{P^2(\Omega)}{|\Omega|}\leq \frac{16\pi}{(j'_{1,1})^2}.$$
Since $(j'_{1,1})^2<3.5$ and $16\pi > 50$, it is enough to prove that
\begin{equation}\label{enough}
\frac{P^2(\Omega)}{|\Omega|}<14.7.
\end{equation}
For that purpose we solve some kind of reverse isoperimetric inequality by maximizing the ratio $P^2(\Omega)/|\Omega|$ among convex sets lying in
the sector defined by the two axis of symmetry. More precisely,  we will bound the perimeter from above and the area from below. Let us assume without loss of generality that the two axis cross at the origin $O$, that the first axis of symmetry is the line $\{x=0\}$, and that the point $A=(0,1)$ belongs to $\partial \Omega$. By assumption, the second axis of symmetry forms the angle $\theta$ with $\{x=0\}$. We denote by $B$ the point on $\partial \Omega$ that belongs to this second axis, and by $\alpha$ the angle $\widehat{OAB}$ as in Fig. \ref{Figure5}.

\begin{figure}[h]                                             
\begin{center}
\includegraphics{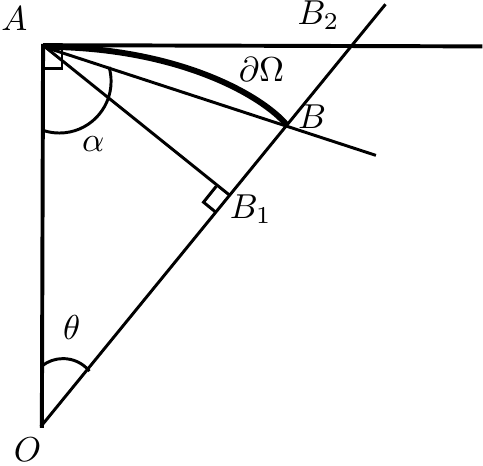}
\caption{\textit{Proof of Proposition \ref{sym5}.}}\label{Figure5}
\end{center}
\end{figure}

Notice that, by convexity, $B$ must lie between the two points $B_1$ and $B_2$ as in Fig. \ref{Figure5}. which provides the bound
$$\frac{\pi}{2}-\theta\leq \alpha \leq \frac{\pi}{2}.$$
We denote by $S$ the positive sector delimited by the half-lines $(OA)$ and $(OB)$, and we consider the restricted sets 
$$\Omega_0:=\Omega\cap S  \quad \text{ and } \quad  \Gamma_0:= \partial \Omega \cap S.$$
By symmetry we know that 
$$|\Omega|=2N |\Omega_0| \quad \text{ and } \quad P(\Omega)= 2 N \mathcal{H}^1(\Gamma_0).$$

The remaining part of the proof is divided into two steps, in which we provide a lower bound for the area and an upper bound for the perimeter, respectively. 

 \emph{Step 1: lower bound for the area.}  To estimate   $|\Omega_0|$ we will use the fact that, by convexity, it has bigger area than the triangle $OAB$, which turns out to be  equal to $\frac{1}{2}x_B$, being $x_B$ the $x$-coordinate of the point $B$. The equation of the line $(AB)$ is
$$
\cos(\alpha) x +\sin(\alpha) y- \sin(\alpha)=0,
$$
and $B$ also lies on the line $(OB)$ whose equation is $y=\frac{x}{\tan(\theta)}$. We deduce that 
$$x_B=\frac{\sin(\alpha)\sin(\theta)}{\sin(\theta+\alpha)} \quad \quad  \text{ and }  \quad \quad y_B=\frac{\sin(\alpha)\cos(\theta)}{\sin(\theta+\alpha)}.$$
Therefore,
\begin{eqnarray}
|\Omega|=  N|\Omega_0|\geq N x_B = N  \frac{\sin(\alpha)\sin(\theta)}{\sin(\theta+\alpha)} . \label{lowerbound}
\end{eqnarray}

\emph{Step 2: upper bound for the perimeter.}  Notice that by symmetry, the  perpendicular to the line $(OB)$ passing through $B$ is the tangent line to $\partial \Omega$ at $B$. Similarly, the line perpendicular to $\{x=0\}$ through $A$ is the tangent to $\partial \Omega$ at $A$. These two tangents meet at a point $H$ and, by convexity, $\partial \Omega$ stays below the two. This implies that we can estimate the length of $\Gamma_0$ as $\mathcal{H}^1(\Gamma_0)\leq \overline{AH}+\overline{HB}$, see also Fig. \ref{Figure6}.

\begin{figure}[h]                                             
\begin{center}
\includegraphics{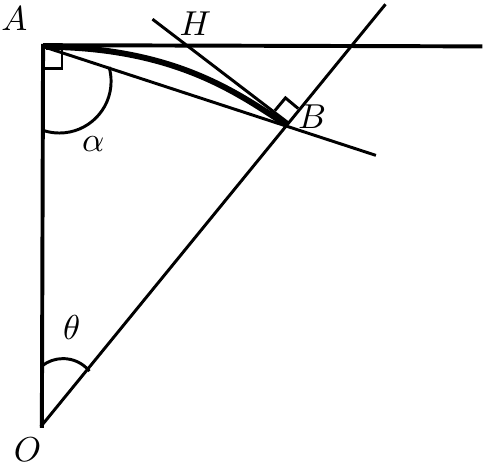}
\caption{\textit{Proof of Proposition \ref{sym5}.}}\label{Figure6}
\end{center}
\end{figure}

Let us determine the $x$-coordinate $x_H$ of $H$ (the $y$-coordinate is by construction 1). The equation of the line $(HB)$ is given by
$$\tan(\theta)(x-x_B)+(y-y_B)=0.$$
Therefore
$$x_H=\frac{y_B-1}{\tan(\theta)}+x_B=\frac{\sin(\alpha)}{\sin(\theta)\sin(\theta+\alpha)}-\frac{\cos(\theta)}{\sin(\theta)} =-\cot(\alpha +\theta),$$
and we deduce that 
$$\overline{AH} = x_H=-\cot(\alpha+\theta).$$
Next,

\begin{eqnarray}
\overline{BH}^2&=&(x_H-x_B)^2+(1-y_B)^2  \notag \\
&=&\frac{1}{\sin^2(\alpha+\theta)}\Big(\cos(\alpha+\theta)+\sin(\alpha)\sin(\theta)\Big)^2+\Big(\sin(\alpha+\theta)-\sin(\alpha)\cos(\theta)\Big)^2  \notag \\
&=& \frac{1}{\sin^2(\alpha+\theta)}\cos^2(\alpha)  \notag
\end{eqnarray}
 from which we deduce 
 $$\overline{BH}=\frac{\cos(\alpha)}{\sin(\theta+\alpha)}.$$
This means that 
\begin{eqnarray}
\mathcal{H}^1(\Gamma_0)\leq \frac{\cos(\alpha)-\cos(\alpha+\theta)}{\sin(\alpha+\theta)}=\frac{2\sin(\frac{\theta}{2})\sin(\alpha+\frac{\theta}{2})}{\sin(\alpha+\theta)}, \label{upperbound}
\end{eqnarray}
and $P(\Omega)={2}N\mathcal{H}^1(\Gamma_0)$. By combining \eqref{lowerbound} and \eqref{upperbound} we get
\begin{eqnarray}
\frac{P^2(\Omega)}{|\Omega|}
 \leq N\frac{{16}\sin^2(\frac{\theta}{2})\sin^2(\alpha+\frac{\theta}{2})}{\sin(\alpha+\theta)\sin(\alpha)\sin(\theta)}={8}N\tan\left(\frac{\theta}{2}\right)\frac{\sin^2(\alpha+\frac{\theta}{2})}{\sin(\alpha+\theta)\sin(\alpha)}. \notag\label{maximiz}
\end{eqnarray}
Now we maximize in the variable $\alpha \in [\frac{\pi}{2}-\theta , \frac{\pi}{2}]$ the function
$$g:\alpha  \mapsto \frac{\sin^2(\alpha+\frac{\theta}{2})}{\sin(\alpha+\theta)\sin(\alpha)}.$$
For this purpose we notice that $g$ can be rewritten in the following form
$$g(\alpha)=\frac{1-\cos(2\alpha+\theta)}{\cos(\theta)-\cos(2\alpha+\theta)},$$
from which we easily see that $g$ is non increasing on $[ \frac{\pi}{2}-\theta, \frac{\pi}{2}-\frac{\theta}{2}]$  and  non decreasing on $[ \frac{\pi}{2}-\frac{\theta}{2}, \frac{\pi}{2}]$. This implies that
\begin{eqnarray}
\max_{\alpha \in[\frac{\pi}{2}-\theta , \frac{\pi}{2}] }g(\alpha)= g\left(\frac{\pi}{2}\right)=g\left(\frac{\pi}{2}-\theta\right)=\frac{\cos^2(\theta/2)}{\cos(\theta)}, \label{claim}
\end{eqnarray}
and we deduce that 
\begin{eqnarray}
\frac{P^2(\Omega)}{|\Omega|} \leq {8}N \tan\left(\theta/2\right) \frac{\cos^2(\theta/2)}{\cos(\theta)}={4}N\tan(\theta)={4}N\tan({\pi/N}). \label{boundD}
\end{eqnarray}
Since $4N\tan\theta$ equals the ratio $P^2(\Omega)/|\Omega|$ for both triangles $OAB_1$ and $OAB_2$ (see Figure \ref{Figure5}) we have actually proved that these two triangles solve
the reverse isoperimetric problem in the sector $S$.

Next, we study the function
$$h:t\mapsto {4}t\tan({\pi/t}).$$
It is easily seen that $h(t)$ is decreasing in $[{2},+\infty[$. Indeed, a direct computation reveals that $h'(t)$ has the same sign as 
\begin{eqnarray}
\sin\left(\frac{\pi}{t}\right)\cos\left(\frac{\pi}{t}\right)-\frac{\pi}{t}. \label{hprime}
\end{eqnarray}
Then for $t\geq {2}$ we have $0\leq \cos(\frac{\pi}{t})\leq 1$ and $0\leq \sin(\frac{\pi}{t}) \leq \frac{\pi}{t}$
from which we deduce that the expression in \eqref{hprime} is non negative, so that $h'(t)\leq 0$ for $t\geq {2}$.
In particular, for every $N\geq 5$, $h(N)\leq h(5)$, thus \eqref{boundD} gives
$$
\frac{P^2(\Omega)}{|\Omega|} \leq h(5) = 20 \tan(\pi/5) < 14.54.
$$
This concludes the proof thanks to \eqref{enough}.
\end{proof}

\section{The case of the axis of symmetry of the equilateral triangle} \label{sectionIII}

In this section we analyze the maximization over the convex sets with the same axis of symmetry of the equilateral triangle.

\begin{proposition}\label{prop-3sym}
Let $\Omega$ be a planar convex domain with two axis of symmetry with angle $\pi /3$. Then the inequality \eqref{ineqConj} holds true. The equality occurs for all and only equilateral triangles.
\end{proposition}

\begin{proof} The proof is detailed in the next pages and occupies the whole section. Let us summarize it here: in Lemma \ref{Ca2} we provide a representation result, by dividing the class of domains under study into a 1-parameter family $\mathcal C_a$, with $a$ varying in the interval $[0,1/3]$, being $a=0$ associated to all and only the equilateral triangles. Then, in \S \ref{ss1} we prove the statement for $a\in [0,1/4]$. Finally, in \S \ref{ss2} we treat the case $a\in [1/4, 1/3]$.
\end{proof}

\subsection*{The family $\mathcal C_a$}
We denote by $T$ the equilateral triangle with side 1 and vertexes at $(\pm 1/2,0)$ and $(0,\sqrt{3}/2)$. For $a\in [0,1/2]$ we denote by $\hat{T}_a$ the equilateral triangle bounded by the lines
$$
y=\frac{\sqrt{3}}{2}(1-a), \quad y= \pm \sqrt{3}x - \sqrt{3} \left(\frac12 - a\right).
$$
These triangles share the same axis of symmetry of $T$, but have the horizontal side on the top instead of the bottom. For the extremal values of $a$, namely $a=0, 1/2$, $\hat{T}_a$ is either circumscribed to $T$ or inscribed into it. We denote by $\widehat{T}$ the inscribed equilateral  triangle $\hat{T}_{1/2}$.
For the intermediate values of $a$, namely $0<a<1/2$, $\hat{T}_a$ crosses $T$. We define the intersection as 
\begin{equation}\label{Oa}
\Omega_a:= \hat{T}_a \cap T
\end{equation} 
that is, the hexagon with vertexes
$$
\left(\pm \frac{a}{2}; \frac{\sqrt{3}}{2}(1-a)\right), \quad 
\left(\pm \left(\frac12 -a\right); 0\right), \quad 
\left(\pm \frac12 (1-a); \frac{\sqrt{3}}{2}a\right).
$$
By connecting the midpoints of the sides $\Omega_a$, we obtain an hexagon, denoted by $H_a$, with vertexes
$$
\left(0; \frac{\sqrt{3}}{2}(1-a)\right), \quad \left( \pm \frac14, \frac{\sqrt{3}}{4}\right), \quad 
\left(\pm \left(\frac12 -\frac34 a\right); \frac{\sqrt{3}}{4} a\right), \quad 
\left(0,0\right).
$$
These domains are represented in Fig. \ref{fig-domains}. 

\begin{figure}[h]                                             
\begin{center}                                                
{\includegraphics[height=3truecm] {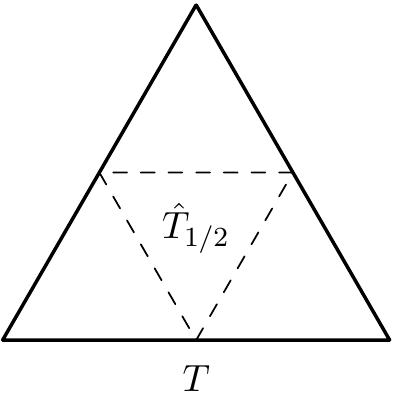}}\quad                                              
{\includegraphics[height=3truecm] {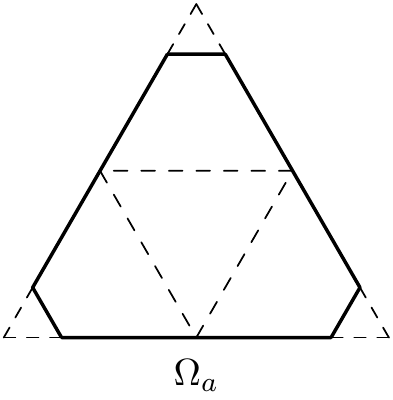}}\quad                                               
{\includegraphics[height=3truecm] {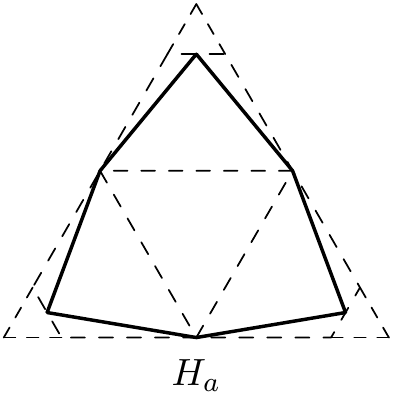}}                                               
\end{center}                                                  
\caption{\textit{The equilateral triangles $T$ and $\hat{T}_{1/2}$, and the hexagons $\Omega_a$ and $H_a$ for $a= 0.15$.}}\label{fig-domains}
\end{figure}

We are now in a position to introduce the subfamily $\mathcal C_a$.
\begin{definition} Let $a\in [0,1/2]$. We say that $\Omega\in \mathcal C_a$ if the following conditions hold:
\begin{itemize}
\item[i)] $\Omega$ is convex and has the same axis of symmetry of $T$;
\item[ii)] $H_a \subset \Omega\subset \Omega_a$.
\end{itemize}
\end{definition}

A priori, $a\in [0,1/2]$. However, as we prove in the next lemma, the interval $[0,1/3]$ is enough to describe all the domains under study.
\begin{lemma}\label{Ca2}
For every convex set $\Omega$ with the same axis of symmetry of the equilateral triangle, there exists $a\in {[0,1/3]}$ such that, up to a rigid motion and a homothety, $\Omega$ belongs to $\mathcal C_a$.
\end{lemma}
\begin{proof}
Let $\Omega$ be as in the assumption. Throughout the proof, with a slight abuse of notation, we will use the same letter for the class of shapes that are obtained by a rigid motion or a homothecy by $\Omega$. Without loss of generality, we may assume that $\Omega$ has the same axis of symmetry of $T$, it lies into the half-plane $y\geq 0$, and it is tangent to the horizontal line $y=0$, as it occurs for $T$. By combining the symmetry assumption on $\Omega$ with the (sharp) bound $y\geq 0$ of its points, we infer that $\Omega$ is contained into $T$ and touches the three segments of $\partial T$ at the midpoints. In particular, by convexity, the triangle $\hat{T}_{1/2}$ is contained into $\Omega$.

Let now $r$ be the horizontal line tangent to $\Omega$ from above. In view of the inclusions $\hat{T}_{1/2}\subset \Omega \subset T$, we infer that $r$ is of the form $y=(1-a)\sqrt{3}/2$ for some $a\in [0,1/2]$. Arguing as above, we infer that $\Omega$ is contained into the equilateral triangle $\hat{T}_a$ and touches the three segments of $\partial \hat{T}_a$ at the midpoints. All in all, we infer that $\Omega$ is contained into $T\cap \hat{T}_a=\Omega_a$ and that it touches $\partial \Omega_a$ at the midpoints of the boundary segments, which by definition are the vertexes of $H_a$. By convexity, $H_a\subset \Omega$. In other words $\Omega\in \mathcal C_a$.

In order to conclude the proof, we need to show that $a$ can be taken in $[0,1/3]$ instead of $[0,1/2]$. To this aim we use a very simple trick, which, roughly speaking, consists in changing the role of $T$ and $\hat{T}$. The class $\mathcal C_a$ is constructed starting from the hexagon $\Omega_a$, which is the intersection between $T$ and $\hat{T}_a$. The parameter $a$ is nothing but the length of the portions of $\partial \hat{T}_a$ included into $T$. The same hexagon can be described in terms of the length, say $b$, of the portions of $\partial T$ included into $\hat{T}_a$. In other words, changing the role of $T$ and $\hat{T}_a$, we may write $\Omega_a$ as an $\Omega_b$, for $b=1-2a$. If a shape $\Omega$ belongs to the class $\mathcal C_a$ associated to $T$, then the rescaled shape $(1/(2-3a))\Omega$ belongs to $\mathcal C_{(1-2a)/(2-3a)}$ associated to the triangle $(1/(2-3a))\hat{T}_a$. The triangle $(1/(2-3a))\hat{T}_a$ is equilateral and has sides of length 1. For $a\in [1/3, 1/2]$, the parameter $(1-2a)/(2-3a)$ runs from $0$ to $1/3$. This concludes the proof.
\end{proof}

\subsection*{{Description of shapes in $\mathcal C_a$}}
In this paragraph we give some definitions and properties to represent the symmetric shapes of $\mathcal C_a$. 

We start with a useful formula to make integrations over symmetric sets. Let  $\ell_0$ denote the vertical line $x=0$,  $\ell_1$ the line $-x + \sqrt{3} y = 1/2$, and $\ell_2$ the line $x + \sqrt{3} y = 1/2$. Let $\sigma_1$ and $\sigma_2$ denote the symmetries with respect to $\ell_1$ and $\ell_2$, respectively: they are given by the transformations
\begin{equation}\label{sigmai}
\sigma_1: \begin{cases}
x' = \frac12 x + \frac{\sqrt{3}}{2} y -\frac14
\\
y' = \frac{\sqrt{3}}{2} x - \frac12 y + \frac{\sqrt{3}}{4}
\end{cases}\quad
\sigma_2: \begin{cases}
x' = \frac12 x - \frac{\sqrt{3}}{2} y +\frac14
\\
y' = -\frac{\sqrt{3}}{2} x - \frac12 y + \frac{\sqrt{3}}{4}.
\end{cases}
\end{equation}

Any set $E$ with the axis of symmetry $\ell_0, \ell_1, \ell_2$ can be written as the disjoint union of three components: $\tilde{E}  \cup \sigma_1(\tilde{E}) \cup \sigma_2(\tilde{E})$, for some set $\Tilde{E}$ symmetric with respect to the vertical axis. As a consequence, we get for every function $f$
\begin{equation}\label{trick}
\int_{E} f = \int_{\tilde{E} \cup \sigma_1(\tilde{E}) \cup \sigma_2(\tilde{E})} f = \int_{\tilde{E}} (f + f\circ \sigma_1 + f\circ \sigma_2). 
\end{equation}
This formula applies in particular to the elements of $\mathcal C_a$.

\medskip

Let us exploit again the three axis of symmetry to represent the elements of $\mathcal C_a$. We start from $\Omega_a$: it is the union of the hexagon $H_a$ and 6 copies (obtained by reflection) of the triangle $ABC$, with vertexes
\begin{equation}\label{ABC}
A=\left(0, \frac{\sqrt{3}}{2}(1-a) \right),\quad 
B=\left(\frac{a}{2}, \frac{\sqrt{3}}{2}(1-a) \right), \quad 
C=\left(\frac14, \frac{\sqrt{3}}{4} \right).
\end{equation}
Any other element $\Omega \in \mathcal C_a$ is the union of $H_a$ and 6 copies (reflections) of the convex set
\begin{equation}\label{omega}
\omega:=\Omega \cap ABC.
\end{equation}
Of course $\omega=\emptyset$ if $\Omega= H_a$. The perimeter of $\Omega$ is nothing but
\begin{equation}\label{Pomega}
P(\Omega) = 6 [P(\omega) - \overline{AC}].
\end{equation}
Going back to \eqref{trick}, we deduce an alternative formula of integration in the case of $f$ even in $x$, i.e., $f(-x,y)=f(x,y)$:
\begin{equation}\label{trick2}
\int_{\Omega} f = \int_{H_a} f + 2 \int_\omega (f + f\circ \sigma_1 + f \circ \sigma_2).
\end{equation}

\subsection{The case $a\in [0,1/4]$}\label{ss1}

As announced in the Introduction, this is the more delicate case to handle. 
The difficulty comes from calculations; on the other hand, the strategy is simple to present. Given $\Omega\in \mathcal C_a$, we look for three objects: an approximating convex set from outside $\Omega_1 \supset \Omega$, an approximating convex set from inside $\Omega_2 \subset \Omega$, and a function $v\in C^\infty(\mathbb R^2)$ with zero average on $\Omega$. By construction, we have
\begin{equation}\label{firstub}
P^2(\Omega)\mu_1(\Omega) \leq P^2(\Omega_1) \frac{\int_{\Omega_1} |\nabla v|^2}{\int_{\Omega_2} v^2}.
\end{equation}
The goal is to choose the three objects in a smart way, making this upper bound less than or equal to $16\pi^2$. 

\medskip

We start with the choice of the test function. We take
\begin{equation}\label{vk}
v_k:=(1-k)u_1+k\hat{u}_1,
\end{equation}
with  $k\in [0,1]$, which will be chosen later, with $u_1$ one of the first eigenfunctions of $T$ associated to $\mu_1(T)$:
\begin{equation}\label{u1}
u_1(x,y)= \sin\left(\frac43 \pi x \right) +  2 \cos\left(\frac{2}{\sqrt{3}}\pi y \right) \sin \left( \frac23 \pi x\right),
\end{equation}
and with $\hat{u}_1$ one of the first eigenfunctions of $\widehat{T}$ associated to $\mu_1(\hat{T})$:
\begin{equation}\label{u1hat}
\hat{u}_1(x,y)=\sin\left(\frac{8}{3} \pi x\right) - 2\cos\left(\frac{4}{\sqrt{3}} \pi y\right)
\sin\left(\frac{4}{3} \pi x\right).
\end{equation}
Note that $u_1(-x,y)=-u_1(x,y)$ and $\hat{u}_1(-x,y)=-\hat{u}_1(x,y)$, therefore $v_k$ is odd too in the $x$ variable and has zero average in $\Omega$, which is symmetric with respect to the $y$ axis. This ensures that $v_k$ is an admissible test function for $\mu_1(\Omega)$.

Let us now pass to the construction of $\Omega_1$ and $\Omega_2$. Let $\omega$ be the set associated to $\Omega$ according to \eqref{omega}, namely $\omega=\Omega \cap ABC$. Take now the line parallel to $AC$ tangent to $\omega$. This line intersects the sides $AB$ and $BC$ in two points (possibly coinciding):  we call $Q_1$ the point on the segment $AB$ and $Q_2$ the point on the side $BC$. We introduce a parameter to describe these points: $Q_2$ is of the form
$$
Q_2(c)=\left( \frac{c}{2}; \frac{\sqrt{3}}{2} (1-c)\right), \quad \hbox{for some }  c\in [a,1/2].
$$
Note that $B=Q_2(a)$ and $C=Q_2(1/2)$. Using the expression of $Q_2$ we can deduce the coordinates of $Q_1$ as functions of $c$. 
Let us now consider one of the tangent points of $Q_1(c)Q_2(c)$ to $\omega$: it is of the form 
$$
Q(s)=(1-s)Q_1(c) + s Q_2(c), \quad s\in [0,1].
$$
\begin{figure}[h]                                           
\begin{center}                        
{\includegraphics[height=5truecm] {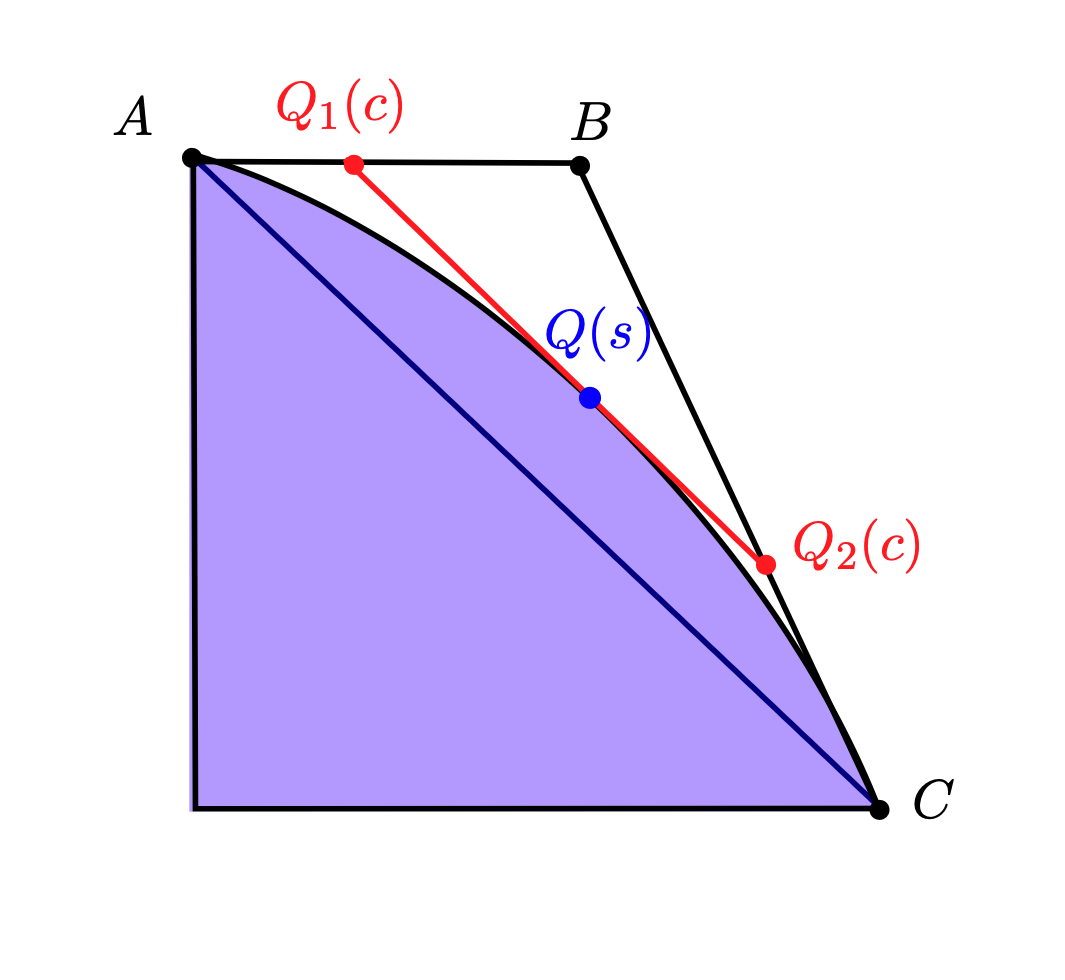}}                               
\caption{The parameters $c$ and $s$}\label{figuretets1}                                      
\end{center}                                                  
\end{figure}    
These definitions are represented in Figure \ref{figuretets1}. By convexity $\omega$ contains the triangle $AQ(s)C$ and is contained into the trapezoid $AQ_1(c)Q_2(c)C$. In particular, we may choose as $\Omega_1$ the union of $H_a$ and 6 copies (rotations) of the trapezoid $AQ_1(c)Q_2(c)C$, and as $\Omega_2$ the union of $H_a$ and 6 copies (rotations) of the triangle $AQ(s) C$.
With this choice of $\Omega_1$, using the perimeter formula \eqref{Pomega}, we obtain
$$
P(\Omega) \leq P(\Omega_1) = 6 [  \overline{AQ_1(c)} + \overline{Q_1(c)Q_2(c)} + \overline{Q_2(c)C} ].
$$
By applying the integration formula \eqref{trick2} to $f=|\nabla v_k|^2$ and then $f=|v_k|^2$, satisfying both $f(-x,y)=f(x,y)$, we get
\begin{align*}
& \int_{\Omega_1} |\nabla v_k|^2 = \int_{H_a} |\nabla v_k|^2  + 2 \int_{AQ_1(c)Q_2(c)C} \left[|\nabla v_k|^2 + |\nabla v_k|^2\circ \sigma_1 + |\nabla v_k|^2\circ \sigma_2\right],
\\
& \int_{\Omega_2} v_k^2 =  \int_{H_a} v^2  + 2 \int_{AQ(s)C}\left[ |v_k|^2 + |v_k|^2\circ \sigma_1 + |v_k|^2\circ \sigma_2\right].
\end{align*}
By combining these results, setting for brevity
\begin{align}
V & :=|\nabla v_k|^2 + |\nabla v_k|^2\circ \sigma_1 + |\nabla v_k|^2\circ \sigma_2, \label{defV0}
\\
U & := |v_k|^2 + |v_k|^2\circ \sigma_1 + |v_k|^2\circ \sigma_2,\label{defU0}
\end{align}
we can rewrite the bound \eqref{firstub} as follows:
\begin{equation}\label{estimate}
\begin{split}
&{P^2(\Omega) \mu_1(\Omega)}
\\
&{\leq \left[6\, \left(  \overline{AQ_1(c)} + \overline{Q_1(c)Q_2(c)} + \overline{Q_2(c)C} \right)\right] ^2 \frac{\int_{H_a} |\nabla v_k|^2  + 2 \int_{AQ_1(c)Q_2(c)C} V}{ \int_{H_a} |v_k|^2  + 2 \int_{AQ(s)C}U}.}
\end{split}
\end{equation}

\medskip
We explain now our strategy to prove that the right-hand side of \eqref{estimate} is less than or equal to $16\pi^2$:

{\bf 1st step.} We prove that the integral $\int_{AQ(s)C}U$ appearing in the denominator is decreasing in $s$, at least when $0\leq k\leq 1/6$. 
Therefore, to find the upper bound, we can choose $s=1$, namely $Q(s)=Q_2(c)$. More precisely, denoting by $\widetilde{P}(a,c):=6\, \big(\overline{AQ_1(c)} + \overline{Q_1(c) Q_2(c) } + \overline{Q_2(c) C} \big)$, $N(a,c):=\int_{H_a} |\nabla v_k|^2  + 2 \int_{AQ_1(c)Q_2(c)C} V$, and $D(a,c):=\int_{H_a} v_k^2  + 2 \int_{AQ_2(c)C}U$, we obtain
$$
P^2(\Omega)\mu_1(\Omega) \leq \left(\widetilde{P}(a,c)\right)^2 \frac{N(a,c)}{D(a,c)}.
$$
Note that the upper bound is less than or equal to $16 \pi^2$ if and only if
\begin{equation}\label{defF}
F(a,c):=\left(\widetilde{P}(a,c)\right)^2 N(a,c) -16 \pi^2 D(a,c) \leq 0.
\end{equation}

\smallskip

{\bf 2nd step.}
By precise estimates, we prove that $F(a,c) \leq 0$ for all $a\leq 1/60$ (and for all $c$) or for all $c\leq 0.16$ (and for all $a$) if we choose $k=0$, namely as a test function $v_0=u_1$.

\smallskip

{\bf 3rd step.} For the remaining cases, $a\in [1/60,1/4]$, we provide a more numerical proof. First we compute exactly all the integrals occurring in  
$N(a,c)$ and $D(a,c)$. Then, by estimating the partial derivatives $\frac{\partial F}{\partial a}$
and $\frac{\partial F}{\partial c}$, we are led to compute $F(a,c)$ only in a finite number of points, namely 530\,000 points in the sector $a\in [1/60,1/4]$, $c\in [\max(a,0.16),1/2]$.
For that purpose, we use successively as test functions $v_k$ for $k=0$, $k=0.06$ and $k=0.12$.
We prove that way the desired inequality $F(a,c)<0$.

The most technical parts of this program are postponed to the Appendix.

\subsection*{1st step: getting rid of the parameter $s$.} We start with some preliminary results on the functions $U,V$ and on their integrals over $T$ and $\Omega_a$. 

In view of definition \eqref{vk}, the functions $V$ and $U$ introduced in \eqref{defV0}-\eqref{defU0} are of the form
\begin{eqnarray}
&V=(1-k)^2 V_1 +k^2 V_2 +2k(1-k) V_3,\label{defVbis}
\\
&U=(1-k)^2 U_1 +k^2 U_2 +2k(1-k) U_3,\label{defUbis}
\end{eqnarray}
with
\begin{equation}\label{defVall}
\begin{array}{l}
V_1(x,y):= |\nabla u_1|^2 (x,y) + |\nabla u_1|^2( \sigma_1(x,y)) + |\nabla u_1|^2(\sigma_2(x,y)), \\
V_2(x,y):= |\nabla \hat{u}_1|^2 (x,y) + |\nabla \hat{u}_1|^2( \sigma_1(x,y)) + |\nabla \hat{u}_1|^2(\sigma_2(x,y)), \\
V_3(x,y):= \nabla u_1.\nabla \hat{u}_1 (x,y) + \nabla u_1.\nabla \hat{u}_1( \sigma_1(x,y)) +  \nabla u_1.\nabla \hat{u}_1(\sigma_2(x,y)), 
\end{array}
\end{equation}
and
\begin{equation}\label{defUall}
\begin{array}{l}
U_1(x,y):=u_1^2 (x,y) + u_1^2( \sigma_1(x,y)) + u_1^2(\sigma_2(x,y)) , \\
U_2(x,y):=\hat{u}_1^2 (x,y) + \hat{u}_1^2( \sigma_1(x,y)) + \hat{u}_1^2(\sigma_2(x,y)) , \\
U_3(x,y):=u_1 \hat{u}_1 (x,y) + u_1 \hat{u}_1( \sigma_1(x,y)) + u_1 \hat{u}_1(\sigma_2(x,y)) .\\
\end{array}
\end{equation}
\begin{lemma} \label{lemma8.2}
Let $V_i$ and $U_i$, $i=1,2,3$ be the functions defined in \eqref{defVall}-\eqref{defUall}, with $\sigma_1$ and $\sigma_2$ the symmetries introduced in \eqref{sigmai}. 
Let us introduce the following functions:
\begin{equation}\label{fipsieta}
\begin{array}{l}
\varphi(x,y)=\cos \frac{4\pi y}{\sqrt{3}}  - 2 \cos \frac{2\pi y}{\sqrt{3}} \cos 2\pi x, \\
\psi(x,y)=\cos 4\pi x - 2\cos \frac{6\pi y}{\sqrt{3}}  \cos 2\pi x, \\
\eta(x,y)=\cos \frac{8\pi y}{\sqrt{3}} + 2 \cos \frac{4\pi y}{\sqrt{3}}  \cos 4\pi x. 
\end{array}
\end{equation}
Then
\begin{eqnarray*}
 V_1(x,y) &=& \frac{8\pi^2}{3} \left( 3 - \varphi(x,y)\right),\label{V1}\\
 V_2(x,y) &=& \frac{32\pi^2}{3} \left( 3 - \eta(x,y)\right),\label{V2}\\
 V_3(x,y) &=& \frac{16\pi^2}{3} \left( \psi(x,y) - \varphi(x,y)\right),\label{V3}\\
 U_1(x,y) &=& 3 \left( \frac32 + \varphi(x,y)\right), \label{U1}\\
 U_2(x,y) &=& 3 \left( \frac32 + \eta(x,y)\right), \label{U2}\\
 U_3(x,y) &=& -3 \left( \frac12 \psi(x,y) + \varphi(x,y)\right). \label{U3}\\
\end{eqnarray*}
\end{lemma}
\begin{proof}
The proof is straightforward. For example,  for $u_1$
using the expressions of the symmetries $\sigma_1, \sigma_2$ given in \eqref{sigmai}, we get
\begin{equation}\label{u1sig1}
u_1(\sigma_1(x,y))=-u_1(x+1,y)\quad u_1(\sigma_2(x,y))=-u_1(x-1,y) .
\end{equation}
Now since $u_1^2$ can be written
$$u_1^2(x,y)=\frac{1}{2}-\frac{1}{2} \cos \frac83 \pi x  +2 \cos \frac{2\pi y}{\sqrt{3}} \big(\cos \frac23 \pi x  - \cos 2\pi x \big)+
\big(1+\cos \frac{4\pi y}{\sqrt{3}} \big)\big(1-\cos \frac43 \pi x \big)$$
the formula for $U_1$ follows thanks to the property $\cos(a)+\cos(a+2\pi/3)+\cos(a+4\pi/3)=0$ that holds  for any $a$.\\
The same argument is used for all the other functions.
\end{proof}

The first application of the previous lemma is the monotonicity of the integrand $U$, which allows us to get rid of the parameter $s$.
\begin{lemma}\label{lemmas}
Let $U$ be the function defined in \eqref{defUbis}, namely $U=(1-k)^2 U_1 +k^2 U_2 +2k(1-k) U_3$ with $U_i$ introduced in \eqref{defUall}. Then, for all $k\in [0,1/6]$ its derivatives satisfy
$$\frac{\partial U}{\partial x} \leq 0 \qquad \frac{\partial U}{\partial y} \geq 0.$$
Consequently, the function $s\mapsto \int_{AQ(s)C} U(x,y)\dx\dy$ is decreasing and, for an upper bound of the Rayleigh quotient, we can choose $s=1$,
namely $Q(s)=Q_2(c)$.
\end{lemma}
The proof is postponed to the Appendix.

We conclude the paragraph with some useful formulas, which are deduced by combining the integration formula \eqref{trick} together with the representation result of Lemma \ref{lemma8.2}.
\begin{lemma}\label{lem-F12}
\begin{align*}
\int_{T\setminus \Omega_a} u_1^2 \mathrm{d}x\mathrm{d}y & = \frac{9 \sqrt{3}a^2}{8} + 
\frac{3\sqrt{3}}{8\pi^2}(1-\cos 2\pi a + 2\pi a \sin 2\pi a)=:F_1(a), \\
\int_{T\setminus \Omega_a} |\nabla u_1|^2 \mathrm{d}x\mathrm{d}y  &= 2\sqrt{3} \pi^2 a^2 - \frac{\sqrt{3}}{3}(1-\cos 2\pi a + 2\pi a \sin 2\pi a)=:F_2(a),\\
\int_{T\setminus \Omega_a} \hat{u}_1^2 \mathrm{d}x\mathrm{d}y & = \frac{9 \sqrt{3}a^2}{8} + 
\frac{3\sqrt{3}}{32\pi^2}(1-\cos 4\pi a + 4\pi a \sin 4\pi a)=:F_3(a),\\
\int_{T\setminus \Omega_a} |\nabla \hat{u}_1|^2 \mathrm{d}x\mathrm{d}y  &= 8\sqrt{3} \pi^2 a^2 - \frac{\sqrt{3}}{3}(1-\cos 4\pi a + 4\pi a \sin 4\pi a)=:F_4(a),\\
\int_{T\setminus \Omega_a} u_1 \hat{u}_1 \mathrm{d}x\mathrm{d}y & = \frac{3\sqrt{3}}{16\pi^2}(\cos 2\pi a +\cos 4\pi a -2 -4\pi a \sin 2\pi a)=:F_5(a), \\
\int_{T\setminus \Omega_a} \nabla u_1 .\nabla \hat{u}_1 \mathrm{d}x\mathrm{d}y & =\frac{2\sqrt{3}}{3}(2\cos 2\pi a -1 -\cos 4\pi a - 2\pi a \sin 2\pi a)=:F_6(a).
\end{align*}
\end{lemma}
\begin{proof}
Let $T_a$ be the triangle located at the north of $T$, namely the triangle with vertexes $\left(\pm \frac{a}{2}; \frac{\sqrt{3}}{2}(1-a)\right)$ and $(0,\frac{\sqrt{3}}{2})$.
According to the symmetry property \eqref{trick}, we have for example
$$\int_{T\setminus \Omega_a} u_1^2\dx\dy = \int_{T_a} U_1(x,y)\dx\dy, \quad \mbox{and} \quad \int_{T\setminus \Omega_a} |\nabla u_1|^2\dx\dy = \int_{T_a} V_1(x,y)\dx\dy, $$
and Lemma \ref{lem-F12} follows from a straightforward computation of these integrals on the triangle $T_a$ using Lemma \ref{lemma8.2}. The computation is the same for the other functions.
\end{proof}

\begin{remark}\label{remT}
Without making use of the reflections, we can integrate over the whole $T$ the following quantities:
\begin{equation}\label{Ray-u1} 
\int_T |\nabla u_1|^2 \mathrm{d}x\mathrm{d}y = \frac{2\pi^2}{\sqrt{3}}, \quad \int_T u_1^2 \mathrm{d}x\mathrm{d}y = \frac{3\sqrt{3}}{8},
\end{equation}
\begin{equation}\label{Ray-u1hat} 
\int_T |\nabla \hat{u}_1|^2 \mathrm{d}x\mathrm{d}y = \frac{8\pi^2}{\sqrt{3}}, \quad \int_T \hat{u}_1^2 \mathrm{d}x\mathrm{d}y = \frac{3\sqrt{3}}{8},
\end{equation}
\begin{equation}\label{mixte}
\int_T \nabla u_1 \cdot \nabla \hat{u}_1\mathrm{d}x\mathrm{d}y=0 , \quad
\int_T u_1\hat{u}_1\mathrm{d}x\mathrm{d}y=0.
\end{equation}
Notice that the former allow us to recover $\mu_1(T)={\int_T |\nabla u_1|^2}/{\int_T u_1^2} = 16 \pi^2 /9$. 
\end{remark}

\subsection*{2nd step: small values of $a$.}
We want to prove that, choosing the test function $u_1$, the quantity $F(a,t)$ defined in \eqref{defF} is negative for small values of $a$, $a \neq 0$ (whereas it vanishes for $a=0$). We are going to prove

\begin{lemma}\label{smalla}
The function $F$ defined in \eqref{defF} satisfies
$$F(a,c) \leq 0\quad \mbox{ when $a\in [0,1/60]$ (for any $c$) or when $c\in [0,0.16]$ (for any $a$)}.$$
Equality occurs only for $a=0$, namely for the equilateral triangle.
\end{lemma}
The proof of the Lemma is postponed to the Appendix.

\subsection*{3rd step: computer assisted proof}

In view of Lemma \ref{smalla}, it remains us to prove the inequality $F(a,c) < 0$ in the zone
\begin{equation}\label{defZ}
\mathcal{Z}:=\{(a,c)\ :\  1/60 \leq a \leq 1/4,\,  \max(a,0.16) \leq c \leq 1/2\}.
\end{equation}
\begin{figure}[h!]
\includegraphics[scale=0.25]{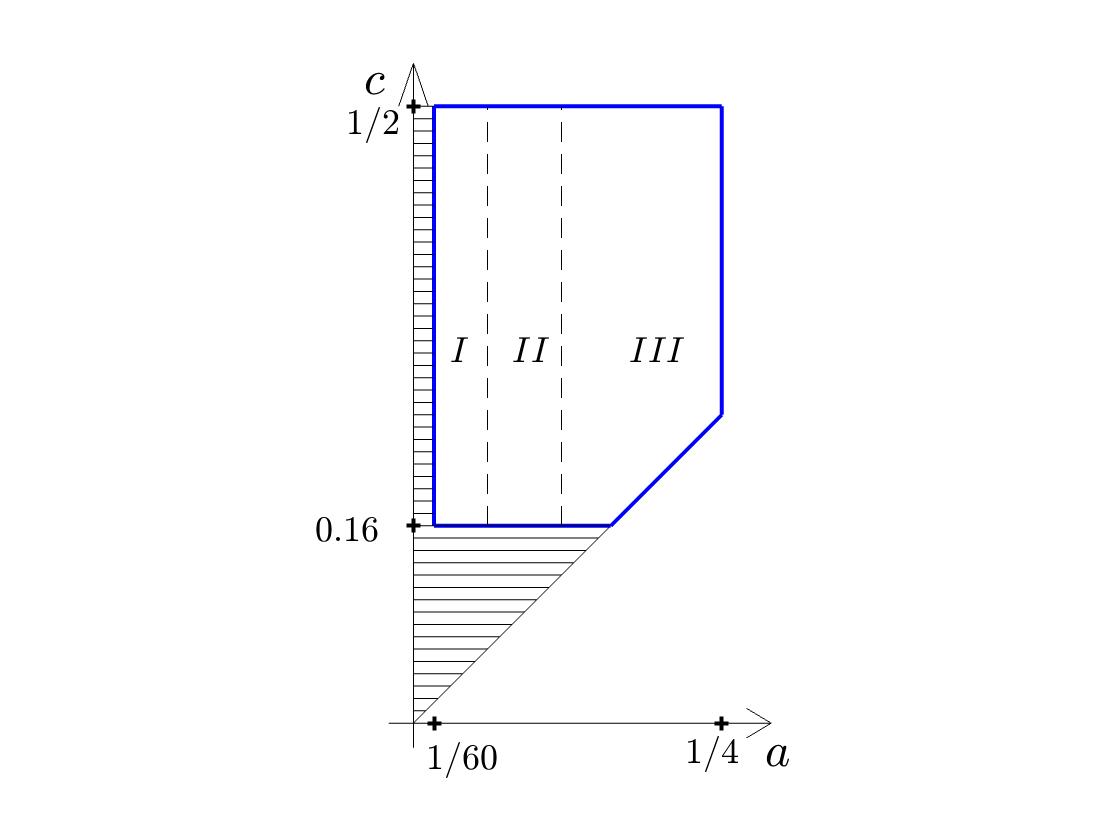}
\caption{The zone $1/60 \leq a \leq 1/4,  \max(a,0.16) \leq c \leq 1/2$ and the three sub-zones $\mathcal Z_{I}$, $\mathcal Z_{II}$, $\mathcal Z_{III}$.}\label{figZ}
\end{figure}

To this aim, we divide $\mathcal Z$ into three subregions, according to the value of $a$, see also Fig. \ref{figZ}:
\begin{equation}\label{defZi}
\begin{array}{l}
\mathcal Z_I := \mathcal Z \cap \{a\in [1/60, 0.06]\},
\\
\mathcal Z_{II} := \mathcal Z \cap \{a\in [0.06, 0.12]\},
\\
\mathcal Z_{III} := \mathcal Z \cap \{a\in [0.12, 0.25]\}.
\end{array}
\end{equation}

In each regime, we choose different values for the parameter $k$:
$$
\begin{cases}
k=0 \quad & \hbox{in }\mathcal Z_I
\\
k=0.06 \quad & \hbox{in }\mathcal Z_{II}
\\
k=0.12 \quad & \hbox{in }\mathcal Z_{III}.
\end{cases}
$$
We recall that $k$ appears in the definition \eqref{vk} of the test function $v_k$, and thus in the definition \eqref{defF} of $F$. In the next lemma we bound the modulus of the partial derivatives of $F$ (with respect to $a$ and $c$) in each of the three regimes.

\begin{lemma}\label{estimategrad} Let $F(a,c)$ be the function introduced in \eqref{defF} and let $\mathcal Z_{I}$, $\mathcal Z_{II}$, $\mathcal Z_{III}$ be the regions defined in \eqref{defZi}. Then
$$
\left| \frac{\partial F}{\partial a} (a,c) \right| \leq 
\begin{cases}
819.6011 \quad & \hbox{in }\mathcal Z_{I}
\\
1048.9639 \quad & \hbox{in }\mathcal Z_{II}
\\
1353.8951 \quad & \hbox{in }\mathcal Z_{III}
\end{cases}
, \quad \left| \frac{\partial F}{\partial c}(a,c) \right| \leq \begin{cases}
84.4817 \quad & \hbox{in }\mathcal Z_{I}
\\
170.9884 \quad & \hbox{in }\mathcal Z_{II} 
\\
352.1112 \quad & \hbox{in }\mathcal Z_{III}
\end{cases}.
$$
\end{lemma}
The proof of the Lemma is postponed to the Appendix.

Let us now explain the computer assisted part. We start by considering the rectangular regions $\mathcal Z_I$ and $\mathcal Z_{II}$, which are of the form $[a_{min}, a_{max}]\times[c_{min}, c_{max}]$. Let for brevity $\mathcal R$ denote one of these rectangles. 
Given $n_a, n_c\in \mathbb N$ and setting 
$$
\delta_a=\frac{a_{max}-a_{min}}{n_a}, \ \delta_c=\frac{c_{max}-c_{min}}{n_c},
$$
the family of points
$$
\{(a_i,c_j)\}_{i=0,\ldots, n_a\,,j=0, \ldots, n_c}, \quad a_i= a_{min} + i\delta_a, \ c_j= c_{min} + j \delta_c,
$$
forms a grid in $\mathcal R$ and defines a uniform partition of $\mathcal R$ into $n_a \cdot n_c$ rectangles.
Given a pair $(a,c)\in \mathcal R$, there exist $i,j$ such that $|a-a_i| \leq \delta_a/2$, $|c-c_j|\leq \delta_c/2$.
Therefore, we deduce that for a pair $(a,c)\in \mathcal R$ and the associated grid point $(a_i,c_j)\in \mathcal R$, there holds
\begin{equation}\label{boundF}
F(a,c) \leq F(a_i, c_j) + \|\partial_a F\|_\infty \delta_a/2 + \|\partial_c F\|_\infty \delta_c/2.
\end{equation}
Denoting by $\mathcal E$ the remainder 
$$
\mathcal E:=  \|\partial_a F\|_\infty \delta_a/2 + \|\partial_c F\|_\infty \delta_c/2,
$$
we deduce from \eqref{boundF} that for every $(a,c) \in \mathcal R$,
\begin{equation}\label{boundF2}
F(a,c) \leq \max_{(a_i, c_j )\in \mathcal R}F(a_i, c_j) + \mathcal E.
\end{equation}

The same conclusion holds true for the region $\mathcal Z_{III}$, but the grid has to be slightly modified: we start from the grid associated to the larger rectangle $[a_{min}, a_{max}]\times[c_{min}, c_{max}]$; then
we refine the partition of the rectangles which cross the bisector line $c=a$, by adding the midpoints of the boundary segments and the intersection point of the diagonals; finally, we remove the points $(a_i, c_j)$ that fall outside $\mathcal Z_{III}$, namely the ones for which $c_j < \max(a_i, 0.16)$. The second operation is needed in order to guarantee that every point of $\mathcal Z_{III}$ stays at distance at most $\delta_a/2$ from one of the $a_i$s and at most $\delta_c/2$ from one of the $c_j$s.

We now bound from above the two terms appearing in the right-hand side of \eqref{boundF2}, by choosing different values of $n_a$, $n_c$, thus of $\delta_a$ and $\delta_c$: in $\mathcal Z_{I}$ we take $(n_a,n_c)=(155,150)$, in $\mathcal Z_{II}$ we take $(n_a,n_c)=(160,155)$, and in the rectangle containing $\mathcal Z_{III}$ we take $(n_a, n_c)=(730,735)$. This corresponds to a non-uniform grid of $\mathcal Z$ of about 530\,000 points (we recall that in $Z_{III}$ we neglect the points in the half plane $c<a$).
Using the upper bounds found in Lemma \ref{estimategrad}, we get
$$
\mathcal E \leq \begin{cases}
0.21032 \quad & \hbox{in }\mathcal Z_{I}
\\
0.38422 \quad & \hbox{in }\mathcal Z_{II} 
 \\
0.202 \quad & \hbox{in }\mathcal Z_{III}
\end{cases}
$$
The numerical computation in the 530\,000 grid points gives
$$
F(a_i,c_j) 
\leq 
\begin{cases}
-0.21184  \quad & \hbox{in }\mathcal Z_{I} 
\\
-0.39006 \quad & \hbox{in }\mathcal Z_{II} 
\\
-0.20324 \quad & \hbox{in }\mathcal Z_{III}. 
\end{cases}
$$
Using these estimates in \eqref{boundF2}, we conclude that 
$$
F(a,c) \leq \begin{cases}
-0.00152 \quad  & \hbox{in }\mathcal Z_{I}
\\
-0.00584  \quad  & \hbox{in }\mathcal Z_{II}
\\
-0.00124   \quad  & \hbox{in }\mathcal Z_{III}
\end{cases} 
$$
namely $F$ is negative everywhere. This concludes the proof.

\begin{remark} We underline that the numerical computation of $F$ is performed on its explicit expression as a function of $k$, $a$, and $c$, involving only trigonometric functions and square roots. Indeed, the integrals appearing in $N$ and $D$ can be computed as linear combinations of integrals over the whole $T$ of $u_1^2, \hat{u}_1^2, |\nabla u_1|^2, |\nabla \hat{u}_1|^2$ (for the explicit expressions, see Remark \ref{remT}), integrals over $T\setminus \Omega_a$ of  $u_1^2$, $\hat{u}_1^2$, $u_1 \hat{u_1}$, $|\nabla u_1|^2$, $|\nabla \hat{u}_1|^2$, $\nabla u_1\cdot \nabla \hat{u_1}$ (see Lemma \ref{lem-F12}), and integrals over the triangles $T_N=Q_1(c)BQ_2(c)$ and $T_D=ABQ_2(c)$ of $\varphi, \psi, \eta$ (see Remark \ref{remInt} in the Appendix).
\end{remark}

\subsection{The case $a\in [1/4, 1/3]$}\label{ss2}
This case is ruled out by the Szeg\H{o}-Weinberger inequality, already mentioned in \eqref{SzegoW}. This inequality states that the disk maximizes $\mu_1$ among sets of given area, in particular
\begin{equation}\label{ratio}
P^2(\Omega) \mu_1(\Omega) = \frac{P^2(\Omega)}{|\Omega|}\mu_1(\Omega)|\Omega|\leq \frac{P^2(\Omega)}{|\Omega|} \pi (j'_{1,1})^2,
\end{equation}
where $j'_{1,1}$ is the first zero of the derivative of the first Bessel function $J_1$. 

In order to conclude, we need to solve, for every $a\in [1/4, 1/3]$,
$$
\max_{\Omega \in \mathcal C_a} \frac{P^2(\Omega)}{|\Omega|}.
$$
This shape optimization problem is of isoperimetric type, since one can fix the perimeter of the admissible shapes. Contrary to the classical isoperimetric problem, here the area functional is minimized. In this sense, we call such a problem a reverse isoperimetric problem. The study of this kind of functional is interesting in itself, therefore we state a more general result. We also point out that the statement is valid for $a$ in a wider range.

\begin{proposition}[A reverse isoperimetric problem]\label{prop-reverse}
Let $a \in [0,1/2]$. Let $\Psi,\Phi:\mathbb [0,+\infty[ \to \mathbb [0,+\infty[$ be two continuous functions with $\Psi$ increasing. Let $\omega$ be a convex subset of the triangle $ABC$ defined in \eqref{ABC} whose boundary is the union of the segment $AC$ with a curve $\Gamma$. Then the maximum
\begin{equation}\label{form}
\max_{\omega} \left\{\Psi(\mathcal H^1(\Gamma))\Phi(|\omega|)\right\}
\end{equation}
can be searched in the family of triangles with vertexes $A$, $B^*$ and $C$, with $B^*$ belonging either to $AB$ (for $a\leq 1/3$) or to $BC$ (for $a\geq 1/3$).
\end{proposition}

\begin{proof} Let us denote for by $T_{ABC}$ the triangle $ABC$. We first notice that the maximum exists: since by construction the area and the perimeter of admissible configurations are bounded, then the supremum of $\Psi(\mathcal H^1(\Gamma))\Phi(|\omega|)$ is finite. The existence of an optimal configuration follows by combining the compactness of the admissible sets with respect to the complementary Hausdorff convergence, the continuity of perimeter and area in the class, and the continuity of $\Psi$ and $\Phi$ in $\mathbb R^+$.

Without loss of generality, we may search for maximizers in the dense subclass of convex polygons inside $T_{ABC}$, having $A$ and $C$ among their vertexes: a maximizer among polygons is also a maximizer among convex sets. The proof is divided into three steps.

\smallskip

\noindent \textit{Step 1.} Arguing by contradiction, it is easy to see that no vertex of an optimal shape is in the interior of $T_{ABC}$: let $Q_1, Q_2, Q_3$ be three vertexes with $Q_2$ in the interior of $T_{ABC}$. We now perform a {\it parallel chord movement}: consider a small perturbation of the set obtained by sliding the point $Q_2$ on the line parallel to $Q_1Q_3$ passing through $Q_2$: since $Q_2$ is in the interior, the deformation can be done in both directions, either towards $Q_1$ or towards $Q_3$. In both cases, convexity and area are preserved; but, at least in one of the two directions, the perimeter increases, giving the contradiction.

\smallskip

\noindent \textit{Step 2.} The previous step implies that the maximizers have at most 4 vertexes: $A$, $C$, one point $Q_1$ in $AB$ and one point $Q_2$ in $BC$. Let us exclude the case in which $Q_1$ and $Q_2$ are (both) in the interior of the segments $AB$ and $BC$, respectively. They are of the form $Q_1= A + t e_1$, for some $t\in ]0,a/2[$, and $Q_2= C + s \nu$, for some $s\in  ]0,1/2-a[$, where $e_1=(1,0)$ and $\nu=(-1/2, \sqrt{3}/2)$ are the unit vectors in the direction of the segments $AB$ (from $A$ to $B$) and $CB$ (from $C$ to $B$), respectively. Let us perturb the configuration keeping the area constant as follows: we replace $Q_1$ and $Q_2$ with
\begin{equation}
Q_1^\e:= A + (t+\e) e_1, 
\quad
Q_2^\e:=C + (s+\delta(\e))\nu,
\end{equation}
for $\e\in \mathbb R$ small and for a suitable $\delta(\e)$. Imposing that the corresponding shape $\omega_\e$ has the same area, we deduce the value of $\delta(\e)$:
$$
\frac{\mathrm{d}}{\mathrm{d}\e}|\omega_\e|_{\lfloor_{\e=0}} = 0 \quad  \Leftrightarrow \quad  \frac{\mathrm{d}}{\mathrm{d}\e}  \left[\left( \frac{a}{2}-t-\e\right)\left(\frac12 - a -s -  \delta(\e) \right)\right]_{\lfloor_{\e=0}} =0,
$$
so that 
$$
\delta(\e)=-2x\e + o(\e), \quad \hbox{with}\quad x:= \frac{\sigma}{\tau}, \quad \hbox{being}\quad \sigma := \frac14 - \frac{a}{2} - \frac{s}{2}, \quad \tau:= \frac{a}{2}-t.
$$
Notice that $\tau \in ]0,a/2[$, $\sigma \in ]0,1/4 - a/2[$, and $x \in ]0, +\infty[$.
As for the perimeter term, namely the length of the curve $\Gamma_\e$, the following expression holds:
$$
\mathcal H^1(\Gamma_\e) 
=  t + \e + s + \delta(\e) + \sqrt{ \left(\tau + \sigma - \e - \frac12 \delta(\e) \right)^2 + \frac{3}{4}\left( 2\sigma  - \delta(\e)\right)^2}.
$$
Let us now compute the derivative of the perimeter with respect to $\e$, avaluated at $\e=0$:
\begin{align*}
\frac{\mathrm{d}}{\mathrm{d}\e} \mathcal H^1(\Gamma_\e)_{\lfloor_{\e=0}} &  =  1 + \delta'(0) + \frac{  \left[2 (\tau + \sigma) (-1-\delta'(0)/2) + 3 \sigma (-\delta'(0))  \right]  }{2\sqrt{   (\tau + \sigma)^2 + 3 \sigma^2}}
\\ & =  1 - 2 x  + \frac{(2x+1) (2x-1)}{\sqrt{ 1 + 4 x^2 + 2x}} 
\\
& =  (2x-1) \left[ \frac{2x+1}{\sqrt{ 1 + 4 x^2 + 2x}} -1 \right].
\end{align*}
Recalling that $x\in ]0, +\infty[$, we notice that the derivative above vanishes for $x=1/2$, it is negative in $]0,1/2[$ and positive in $]1/2, +\infty[$. In particular, as a function of $x$, the perimeter does not have any maximum in the interior of $\mathbb R^+$. This allows us to conclude that none of these configurations is optimal: indeed, for fixed values of $t$ and $s$, the choice of the sign of $\e$ makes the corresponding $x$ increase of decrease, and at least for one of these cases, the perimeter increases.

\smallskip

\noindent \textit{Step 3.} In view of the previous steps, we infer that, among polygons, maximizers have to be searched in a particular class of triangles: the ones with two vertexes in $A$ and $C$, and the third one either between $A$ and $B$ or between $B$ and $C$. Using the notation introduced above (Step 2), we describe the former type of point with a parameter $t\in [0,a/2]$ and denote it by $Q_1(t)$, and the latter with $s\in [0,1/2-a]$ denoted by $Q_2(s)$. Here we allow $t$ and $s$ to take the extremal values, corresponding to triangles degenerating to the segment $AC$ or coinciding with the whole triangle $T_{ABC}$. In particular, all the possible areas of subsets of $T_{ABC}$ are attained by both families of triangles. 
The maximization problem among polygons becomes scalar and a maximizer exists. In this step we compare the pairs of triangles $AQ_1(t)C$ and $AQ_2(s)C$ having the same area. Writing $Q_1(t)$ in coordinates 
$$
Q_1(t)= \left(t, \frac{\sqrt{3}}{2}(1-a) \right),
$$
and recalling the definition \eqref{ABC} of $A$ and $C$, we infer that the perimeter term $\mathcal H^1(\Gamma)$ of $AQ_1(t)C$ is
$$
p_1(t)  := t + \sqrt{\left(\frac14 - t\right)^2 + \frac34 \left(\frac12 - a \right)^2}.
$$
Similarly, writing $Q_2(s)$ in coordinates
 $$
Q_2(s)= \left(\frac14 - \frac{s}2, \frac{\sqrt{3}}{2} \left( \frac12  +  s \right)\right),
 $$
we deduce that the perimeter term of $AQ_2(s)C$ reads
$$
p_2(s) := s + \sqrt{ \left(\frac14 - \frac{s}{2}\right)^2 + \frac{3}{4} \left( s-\frac12 + a\right)^2}.
$$
Imposing that the areas coincide, we infer that, given $t$, $s$ has to be chosen as $s(t)= t (1-2a)/a$. We underline that, given $a$, this map is a bijection from the interval $[0,a/2]$, in which $t$ varies, and in interval $[0,1/2-a]$, in which $s$ varies. Thus we are led to study the sign of $p_1(t)-p_2(t(1-2a)/a)$: this function depends on two variables, $a$ varying in $]0,1/2[$ (for the extremal cases there is nothing to prove) and $t$ varying in $[0,a/2]$. In order to get rid of the dependence on $a$ of the domain of definition of $t$, let us introduce the variable $y:=2t/a\in [0,1]$. According to this definition, $t=ay/2$ and $s(t)=y (1-2a)/2$. A direct computation shows that the terms of $p_1-p_2$ can be rearranged as follows:
\begin{equation}\label{p1p2}
p_1(ay/2)-p_2(y(1-2a)/2)=\frac14[\varphi(y) - \tilde{\varphi}(y)],
\end{equation}
with
$$
\varphi(y):= 2(3a-1) y +  \sqrt{\left(1 - 2 ay\right)^2 + 3 \left(1 - 2a \right)^2}, \quad$$
$$ \tilde{\varphi}(y):= \sqrt{ \left(1 - y(1-2a) \right)^2 + 3 (1- y)^2(1-2a)^2}.
$$
The function $\tilde\varphi$ is clearly non negative. Let us prove that the function $\varphi$ is non negative, too. If $a\geq 1/3$ this is clearly true. If $a\leq 1/3$, then it is easy to see that $\varphi$ is decreasing in $y$, so that
$$
\varphi(y) \geq \varphi(1) = 2 (3a-1) + 2(1-2a) = 2a \geq 0.
$$
Therefore, we are allowed to compare the squares of these functions:  $\varphi^2(y) \geq \tilde{\varphi}^2(y)$ if and only if
$$
 4 y (1-3a)\left[2 (1-a(1+y)) - \sqrt{\left(1 - 2 ay\right)^2 + 3 \left(1 - 2a \right)^2} \right] \geq 0.
$$
The function in square brackets is always positive, since for every $y\in [0,1]$, $a\in [0,1/2]$, there holds $
(1-a(1+y)) \geq 0$ and
$$[2 (1-a(1+y))]^2 - \left[\left(1 - 2 ay\right)^2 + 3 \left(1 - 2a \right)^2 \right]  = 4a(1-2a)(1-y)\geq 0.
$$
Going back to \eqref{p1p2}, we infer that $p_1(t)>p_2(s(t))$ for $a<1/3$, $p_1(t) \equiv p_2(s(t))$ for $a=1/3$, and $p_1(t)<p_2(s(t))$ in the remaining case, $a>1/3$. For every fixed $a$ we have found an optimal triangle among the convex admissible shapes with the same area. The proof is concluded.
\end{proof}

\begin{lemma}\label{lem-SW}
For every $\Omega \in \mathcal C_a$, $a\in [1/4, 1/3]$, there holds
\begin{equation}\label{ratio2}
\frac{P^2(\Omega)}{|\Omega|}\leq \frac{P^2(H_a)}{|H_a|}.
\end{equation}
\end{lemma}
\begin{proof} Let $\omega$ be the intersection of $\Omega$ with the triangle $ABC$ defined in \eqref{ABC}. Its boundary is the union of the segment $AC$ and a curve, that we denote here by $\Gamma$. By construction, $\Omega$ is the union of $H_a$ and 6 copies (reflections) of $\omega$, so that $|\Omega|= |H_a| + 6 |\omega|$ and $P(\Omega)= 6 \mathcal H^1(\Gamma)$. All in all, we may write
\begin{equation}\label{ratio3}
\frac{P^2(\Omega)}{|\Omega|} = \frac{36 (\mathcal H^1(\Gamma))^2}{|H_a| + 6|\omega|}.
\end{equation}
The right-hand side is of the form \eqref{form} and satisfies the assumptions of Proposition \ref{prop-reverse}. Therefore, recalling that $a\leq 1/3$,
we deduce that the optimal $\Omega$ is a hexagon, associated to an $\omega$ belonging to the 1-parameter family of triangles $AB(t)C$, being $B(t)$ the point of $AB$ at distance $t$ from $A$, $t\in [0,a/2]$. Such triangle has area $|\omega | = \sqrt{3}(1-2a)t/8$ and the length of $\Gamma$ is 
$$
p(t):=\overline{AB(t)} + \overline{B(t)C} = t +  \sqrt{\left( \frac{1}{4} - t \right)^2 + \frac34 \left(\frac12 -a\right)^2}.
$$
Using $|H_a| =\sqrt{3}(2 - 3a)/8$, we deduce that
\begin{equation}\label{maxt}
\frac{P^2(\Omega)}{|\Omega|} \leq 36 \frac{8}{\sqrt{3}} \max_{t\in [0,a/2]} \frac{p(t)^2}{(2 - 3a) + 6 (1-2a) t }.
\end{equation}
Let us prove that the function in the right-hand side is convex. To this aim, we introduce the auxiliary functions
$$
D(t):=(2-3a) + 6 (1-2a) t , \quad T(t):= \frac14 - t, \quad R(t):= \sqrt{T^2(t) + \frac34 \left(\frac12 -a\right)^2}.
$$
According to this notation $p(t)=t + R(t)$ and the function in the right-hand side of \eqref{maxt} is $p^2(t)/D(t)$. Exploiting the fact that
$$
p'(t)=\frac{R(t) - T(t)}{R(t)} = \frac{p(t)-1/4}{R(t)}, 
\quad p''(t)=  p'(t) \frac{R(t)+T(t)}{R^2(t)},
\quad D''(t)=0,
$$
we get
\begin{align}
\frac{\mathrm{d}^2}{\mathrm{d}t^2} \frac{p^2(t)}{D(t)} \notag
&
= \frac{2}{D^3} \left[(p')^2D^2 + p p''D^2 - 2 p p' D' D + (D')^2p^2 \right]\notag
\\
& \geq \frac{2}{D^2} \left[(p')^2D + p p''D - 2 p p' D' \right]\notag
\\
&=  \frac{2p'}{D^2} \left[\frac{p-1/4}{R} D + pD \frac{(R+T)}{R^2} - 2 p D'  \right]\notag
\\
&=  \frac{2p'}{D^2} \left[2p\left(\frac{D}{R} - D' \right)  - D \left(\frac{1}{4R} - \frac{pT}{R^2}\right)   \right].
\label{stima}
\end{align}
Setting $\psi(t):= 2p'(t)/D^2(t)$ and 
$$\phi_1(t):=2\frac{p(t)}{R(t)}\left[D(t) - D'(t)R(t) \right] ,
\quad \phi_2(t):= \frac{D(t)}{4 R^2(t)}\left[R(t) - 4 p(t)T(t)\right],
$$
the last term in \eqref{stima} is $\psi(\phi_1-\phi_2)$. It is immediate to check that $\psi\geq 0$: this is a consequence of $R(t)\geq |T(t)|>T(t)$, giving $p'(t)\geq 0$. In particular, we deduce that 
$$
\frac{\mathrm{d}^2}{\mathrm{d}t^2} \frac{p^2(t)}{D(t)} \geq \psi(t) \left[ \min_{[0,a/2]} \phi_1 - \max_{[0,a/2]} \phi_2\right].
$$
The two functions $\phi_1$ and $\phi_2$ are both increasing, since
\begin{align*}
\phi_1'(t)  & = \frac{2}{R^3} \left[ D (R^2 + t T) + D' R^2 (t+T) \right] \geq 0,
\\
\phi_2'(t) &= \frac{p'}{R^3}  \left[  t (D' R^2 + 2TD) + Dp R  \right] \geq 0.
\end{align*}
This implies that
\begin{align*}
\min_{[0,a/2]} \phi_1 & = \phi_1(0)=2\Big(2-3a-3(1-2a)\sqrt{1-3a+3a^2}\Big).
\\
\max_{[0,a/2]} \phi_2 & = \phi_2(a/2)=\frac{a(1-3a^2)}{(1-2a)}.
\end{align*}
Evaluating $\min {\phi_1} - \max \phi_2$ for $a$ in the range $[1/4, 1/3]$, we infer that the second order derivative of $p^2/D$ is positive, thus the function is convex. In particular, its maximum is either at $0$ or at $a/2$. A direct computation shows that 
$$
\frac{p^2(0)}{D(0)} -  \frac{p^2(a/2)}{D(a/2)} =   \frac{ 1-3a+3a^2}{4(2-3a)} - \frac{(1-a)^2}{8(1-3a^2)} = 
\frac{a (1-2a) (1-3a)^2}{8 (1-3a^2) (2-3a)} >0.
$$
The maximizer $t=0$ corresponds to $\omega=\emptyset$ in the ratio \eqref{ratio3}, associated to $H_a$. This concludes the proof.
\end{proof}

Let us now maximize the right-hand side of \eqref{ratio2}: recalling that 
$$
P(H_a)=3 \sqrt{1 - 3 a + 3a^2}, \quad |H_a| = \frac{\sqrt{3}}{8}(2 - 3a),
$$
and exploiting the decreasing monotonicity of the function
$$
[1/4, 1/3] \ni x\mapsto \frac{1-3x+3x^2}{2-3x},
$$
we infer that $P^2(H_a)/|H_a|\leq P^2(H_{1/4})/|H_{1/4}|$. Thus \eqref{ratio} and \eqref{ratio2} give
\begin{equation}\label{case2}
P^2(\Omega) \mu_1(\Omega) \leq \pi (j'_{1,1})^2  \frac{P^2(H_{1/4})}{|H_{1/4}|} = \pi (j'_{1,1})^2 \frac{42 \sqrt{3}}{5} < 50 \pi,
\end{equation}
where in the last inequality we have used the fact that $(j'_{1,1})^2<3.4$. The last term is below $16\pi^2$, and this gives the proof of Proposition \ref{prop-3sym} in the range $a\in[1/4,1/3]$.

\appendix
\section{Proof of Lemma \ref{lemmas}}
\begin{proof}
Following Lemma \ref{lemma8.2}, the function $U$ is given by
$$U=3\left(\frac{3}{2}[(1-k)^2+k^2] + (1-k)^2 \varphi +k^2 \eta -2k(1-k) \left[\varphi + \frac{\psi}{2}\right]\right),
$$
where $\varphi, \psi, \eta$ are the functions defined in \eqref{fipsieta}. It follows that
$$
 \frac{\partial U}{\partial x}=12\pi g(k;x,y) \sin 2\pi x ,
$$
with
\begin{align*}
 g(k;x,y):= & (1-k)(1-3k)\cos \frac{2\pi y}{\sqrt{3}}-4k^2 \cos \frac{4\pi y}{\sqrt{3}}  \cos 2\pi x 
\\
& + k(1-k)\left[2\cos 2\pi x-\cos \frac{6\pi y}{\sqrt{3}}\right].
\end{align*}
Since $x\in [0,\frac{1}{4}]$, we have to prove that $g(k;x,y)\leq 0$ on the triangle $ABC$ or more generally, on the triangle
$\mathbb{T}$ limited by the lines $x= 0, y=\frac{\sqrt{3}}{4}, y=\sqrt{3}(\frac{1}{2} -x)$. We recall that 
$y\in [\frac{\sqrt{3}}{4}, \frac{\sqrt{3}}{2}]$ and therefore, $\cos \frac{2\pi y}{\sqrt{3}} \leq 0$ on this triangle.

We look first at the dependence in $k$. The derivative of $g(k;x,y)$ with respect to $k$ is of the kind $A_1k+A_0$ with
$$A_1=2\left(3 \cos \frac{2\pi y}{\sqrt{3}} -4 \cos 2\pi x \cos \frac{4\pi y}{\sqrt{3}}   -2\cos 2\pi x+
\cos \frac{6\pi y}{\sqrt{3}}  \right)$$
$$A_0=-4\cos \frac{2\pi y}{\sqrt{3}}  +2\cos 2\pi x-\cos \frac{6\pi y}{\sqrt{3}}  .$$
Letting $C=\cos \frac{2\pi y}{\sqrt{3}} $, $A_0$ can also be written as
$$A_0=-4C^3-C+2\cos 2\pi x$$
and then $A_0\geq 0$ since $C\leq 0$. Moreover, choosing $k=1/6$, we get
$$\frac{A_1}{6} +A_0=-\frac83 C^3 +\frac83 \cos 2\pi x [1-C^2]-C \geq 0.$$
Therefore, since this derivative is affine in $k$, $\frac{\partial g}{\partial k} \geq 0$ for all $k\in [0,1/6]$. Consequently 
$g(k;x,y) \leq g(\frac16;x,y)$ for these values of $k$.
Now, keeping the same notation 
$$36 g(1/6;x,y)=-20C^3-8\cos 2\pi x C^2 +30C+14\cos 2\pi x.$$
The factor in front of $\cos 2\pi x$ being positive, we can estimate $\cos 2\pi x$ by its biggest value on the triangle $ABC$, namely
on the segment $AC$. This means that we can replace $2\pi x$ by $\pi(1-a-\frac{2y}{\sqrt{3}})/(1-2a)$. Now this quantity is decreasing
in $a$ (for $y\geq \sqrt{3}/4$), therefore we can replace $a$ by $1/4$ that yields
$$\cos 2\pi x \leq \cos \pi \frac{1-a-\frac{2y}{\sqrt{3}}}{1-2a}\leq \cos\pi\left(\frac32-\frac{4y}{\sqrt{3}}\right)=
-\sin \frac{4\pi y}{\sqrt{3}}.$$
Thus, it remains to prove that, for $y\in [\frac{\sqrt{3}}{4}, \frac{\sqrt{3}}{2}]$
$$15\cos \frac{2\pi y}{\sqrt{3}} -\sin \frac{4\pi y}{\sqrt{3}} \left[10-4\cos \frac{4\pi y}{\sqrt{3}} \right]
-5\cos \frac{6\pi y}{\sqrt{3}} \leq 0.$$
Setting $Y=2\pi y/\sqrt{3} \in [\pi/2, \pi]$ and $Y=\pi/2 +u$ with $u\in [0,\pi/2]$, we are led to consider
\begin{align*}
& -15\sin(u)+\sin(2u)[10+4 \cos(2u)]-5\sin(3u) 
\\ & =\sin u\left(16\cos^3u-20\cos^2u+12\cos u-10\right)
\end{align*}
the right-hand side can be written as $-2+(\cos u-1)(16\cos^2u-4\cos u+8)$, showing that it is always negative. This finishes the proof
of $\frac{\partial U}{\partial x} \leq 0$.

\medskip
Let us now consider the derivative $\frac{\partial U}{\partial y}$. Setting $C=\cos \frac{2\pi y}{\sqrt{3}}$, it can be written as
$$
\frac{\partial U}{\partial y}=\frac{12\pi}{\sqrt{3}}g_2(k;x,y) \sin \frac{2\pi y}{\sqrt{3}},
$$
with
\begin{align*}
g_2(k;x,y):=&(1-k)(1-3k)\left[\cos 2\pi x-2C\right]-
4k^2\left[4C^3-2C+\cos 4\pi x C\right]
\\ & - 3k(1-k)\cos 2\pi x\left[4C^2-1\right].
\end{align*}
As previously, we start by looking at the dependence in $k$. The derivative of $g_2(k;x,y)$ with respect to $k$ is of the kind $B_1k+B_0$ with
$$B_1=4C(1- 2\cos 4\pi x)+24C^2\cos 2\pi x-32C^3$$
$$B_0= 8C-\cos 2\pi x (1+12C^2).$$
We see that $B_0\leq 0$ and 
$$\frac{B_1}{6}+B_0=\frac{2C}{3}[13-2\cos 4\pi x-8C^2]-\cos 2\pi x[1+8C^2] \leq 0.$$
Therefore, $\frac{\partial g_2}{\partial k} \leq 0$ for all $k\in [0,1/6]$ that yields $g_2(k;x,y)\geq g_2(\frac16;x,y)$. Now we have
$$36 g_2(1/6;x,y)= 30\cos 2\pi x [1-2C^2]-2C[11+2\cos 4\pi x]-16C^3.$$
All the terms of the right-hand side are non negative, except possibly the first one if $C^2>1/2$. Let us consider that case
and estimate then $36g_2$ by
$$36 g_2(\frac16;x,y)\geq  30 [1-2C^2]-18C-16C^3>0\quad \mbox{for } C\in [-1, -1/{\sqrt{2}}]$$
this proves that $\frac{\partial U}{\partial y}  \geq 0$ on the domain.

\medskip
Now let us prove that the integral $I(s):=\int_{AQ(s)C} U(x,y)\dx\dy$ is decreasing in $s$. We can compute the derivative of $I(s)$
using the tool of shape derivative, see \cite[Theorem 5.2.2]{Henrot-Pierre}. It comes
\begin{equation}\label{derivI}
I^\prime(s)=H \int_0^1 u\left[U(u Q(s)+(1-u)C)-U(u Q(s)+(1-u)A)\right] \,\mathrm{d}u
\end{equation}
where $H$ is the height of the triangle $AQ(s)C$ issued from $Q(s)$. Let us observe that, for all $u\in [0,1]$, the two points $u Q(s)+(1-u)C$
and $u Q(s)+(1-u)A$ belong to a parallel to $AC$. Now, since the derivatives $\frac{\partial U}{\partial x}$ and $\frac{\partial U}{\partial y}$
are respectively negative and positive, we see that $U$ decreases along such a parallel, namely $U(u Q(s)+(1-u)C)-U(uQ(s)+(1-u)A) \leq 0$.
This proves that $I^\prime(s)\leq 0$ and the thesis.

\end{proof}

\section{Proof of Lemma \ref{smalla}}
\begin{proof}
We start by computing all the quantities appearing in $F(a,c)$, recalling that here $k=0$ and the test function $v_k$ is the eigenfunction $u_1$. According to Lemma \ref{lemmas}, we have replaced
$Q(s)$ by $Q_2(c)$ in the denominator $D(a,c)$. 
Writing the coordinates of the points $A,Q_1(c),Q_2(c),C$
$$
A=\left(0,\frac{\sqrt{3}(1-a)}{2}\right) ,\quad Q_1(c)=\left(\frac{a(1-2c)}{2(1-2a)},\frac{\sqrt{3}(1-a)}{2}\right) ,\quad
$$
$$
Q_2(c)=\left(\frac{c}{2}, \frac{\sqrt{3}(1-c)}{2}\right), \quad C=\left(\frac{1}{4},\frac{\sqrt{3}}{4}\right),$$
we make the dependence on $a$ and $c$ of $\widetilde{P}(a,c):=6\, [ \overline{AQ_1(c)} + \overline{Q_1(c)Q_2(c)(t)} + \overline{Q_2(c)C} ]$ explicit:
\begin{equation}\label{ptildeexp}
\widetilde{P}(a,c)=\frac{6}{1-2a}\left(\frac{(1-2c)(1-a)}{2} + (c-a) \sqrt{1-3a+3a^2}  \right).
\end{equation}
Now, the integrals defining $N(a,c)$ and $D(a,c)$ involve the symmetric part of $|\nabla u_1|^2$ and $u_1^2$ respectively that are given 
by $V_1$ and $U_1$ (defined in \eqref{defVall}, \eqref{defUall}). The formulas in Lemma \ref{lemma8.2} involve the same function
$\varphi(x,y)$ defined in \eqref{fipsieta}. We will compute the desired integral by writing the domains of integration as
$\Omega_a\setminus (Q_1(c)BQ_2(c))$ and $\Omega_a\setminus (ABQ_2(c))$ respectively. Therefore, by setting $T_N:=Q_1(c)BQ_2(c)$ and $T_D:=ABQ_2(c)$, we write
$$\int_{H_a}+2\int_{AQ_1(c)Q_2(c) C}=\int_{\Omega_a}-2\int_{T_N} \quad\mbox{ and }\quad  \int_{H_a}+2\int_{AQ_2(c)C}=\int_{\Omega_a}-2\int_{T_D}.$$
Thus, we are led to compute the integral of $\varphi$ on both triangles $T_N$ and $T_D$. We set
$$I_1(a,c):=\int_{T_N} \varphi(x,y)\dx\dy, \qquad I_2(a,c):=\int_{T_D} \varphi(x,y)\dx\dy.$$
A direct computation provides
\begin{equation}\label{eqi1}
\begin{array}{c}
I_1(a,c)=\frac{\sqrt{3}}{8\pi^2}\left\lbrace  \frac{c-a}{2a-1} 2\pi a\sin 2\pi a+\frac{1-a}{1-2a} (\cos 2\pi a-\cos 2\pi c) \right.\ldots\\
+\left.\frac{1-2a}{a} \left(1-\cos \frac{2\pi a(c-a)}{1-2a}\right) + \frac{1-2a}{1-a}\left(\cos 2\pi c -\cos \frac{2\pi a(1-a-c)}{1-2a}\right)\right\rbrace.
 \end{array}
\end{equation}
and
\begin{equation}\label{eqi2}
\begin{array}{c}
I_2(a,c)=\frac{\sqrt{3}}{8\pi^2}\left\lbrace -\pi a\sin 2\pi a+\frac{2c-a}{2(c-a)} (\cos 2\pi a-\cos 2\pi c) \right.\ldots\\
+\left. 2(c-a) \left(\frac{1-\cos \pi a}{a} +\frac{\cos 2\pi c -\cos \pi a}{2c-a}  \right)\right\rbrace.
 \end{array}
\end{equation}
Now, taking into account the constant terms in $U_1$ and $V_1$ together with the areas
$$|T_N|=\frac{\sqrt{3}}{4}  \frac{a(c-a)^2}{1-2a}
\quad\mbox{and}\quad |T_D|=\frac{\sqrt{3}}{8} a(c-a),$$
and using $F_2(a)$ and $F_1(a)$ given in Lemma \ref{lem-F12}, we get
\begin{equation}\label{Nac}
\begin{array}{c}N(a,c)=\frac{\sqrt{3}}{3} \left\lbrace 2\pi^2(1-3a^2)+1-\cos 2\pi a \ldots \right.
\\ \left.+2\pi a\sin 2\pi a -12\pi^2 \frac{a(c-a)^2}{1-2a}\right\rbrace
+\frac{16\pi^2}{3} I_1(a,c) \quad 
\end{array}
\end{equation}
and 
\begin{equation}\label{Dac}
\begin{array}{c}
D(a,c)=\frac{3\sqrt{3}}{8}\left\lbrace 1-3a^2-3a(c-a) \ldots\right.
\\
\left. -\frac{1}{\pi^2}\left(1-\cos 2\pi a+2\pi a\sin 2\pi a\right)\right\rbrace -6I_2(a,c).  \quad 
\end{array}
\end{equation}
Now we want to estimate from above each term in $F(a,c)=\widetilde{P}^2(a,c) N(a,c) - 16\pi^2 D(a,c)$ in order to get a bound like
$F(a,c)\leq f_1(c)a+f_2(c)a^2$ (note that $F(0,c)=0$) with $f_1(c)<0$ for all $c$. This will allow us to claim that $F(a,c)\leq 0$
as soon as $a\leq -f_1(c)/f_2(c)$ and it will be enough to bound from below this quantity $ -f_1(c)/f_2(c)$. 

Let us start with $\widetilde{P}(a,c)$. Using the inequality $\sqrt{1-u}\leq 1-u/2$ (for $u\leq 1$), we get
$$
\widetilde{P}(a,c) \leq \frac{3}{1-2a}\left[1-3a-ac+3a^2(1+c-a)\right]\leq 3\left[1-(1+c)a+(1+c) a^2\right].
$$
By squaring, we get
\begin{equation}\label{boundP}
\widetilde{P}^2(a,c) \leq 9\left[1-2(1+c)a+(1+c)(3+c) a^2\right].
\end{equation}
In the sequel we assume that $a\leq a_0$ (later, we will choose $a_0=1/60$ on the one hand and $a_0=1/4$ on the other hand). We obtain successively
\begin{align*}
\bullet &\ \frac{c-a}{2a-1} 2\pi a \sin 2\pi a \leq 0,
\\
\bullet &\ (1-\cos 2\pi c)\left[\frac{1-a}{1-2a} - \frac{1-2a}{1-a}\right] \leq (1-\cos 2\pi c) a \left(2+\frac{3-4a_0}{(1-a_0)(1-2a_0)}\,a\right). 
\end{align*}
Moreover, using $u^2/2-u^4/24 \leq 1-\cos u \leq u^2/2$, we get
\begin{align*}
\bullet & \ \frac{1-a}{1-2a} (\cos 2\pi a -1)\leq -2\pi^2a^2, 
\\
\bullet & \   -  12 \pi^2 \frac{a(c-a)^2}{1-2a} +2\frac{1-2a}{a} \left(1-\cos \frac{2\pi a(c-a)}{1-2a}\right) \leq  - 8 \pi^2 (c^2a - 2ca^2),
\\
\bullet & \ \frac{1-2a}{1-a}\left(1 -\cos \frac{2\pi a(1-a-c)}{1-2a}\right)\leq 2\pi^2 a^2 \frac{(1-a-c)^2}{(1-a)(1-2a)} \leq 2\pi^2 a^2 M(a_0,c), 
\end{align*}
with $M(a_0,c)=\max\left((1-c)^2,\frac{(1-a_0-c)^2}{(1-a_0)(1-2a_0)}\right)$. Here we have used the convexity of the map $a\mapsto (1-a-c)^2/[(1-a)(1-2a)]$. Gathering all these bounds, we finally get
$$
N(a,c)\leq \frac{2 \sqrt{3}}{3} \left\lbrace \pi^2+  a\left(2(1-\cos 2\pi c)-4\pi^2 c^2\right) +A_1(a_0,c) a^2 \right\rbrace
$$
where
$$
A_1(a_0,c)= \frac{3-4a_0}{(1-a_0)(1-2a_0)}\ (1-\cos 2\pi c)+2\pi^2\left(M(a_0,c)-1+4c\right).
$$
Multiplying by the inequality \eqref{boundP}, we finally obtain
$$
\widetilde{P}^2(a,c) N(a,c)\leq 6\sqrt{3}\left\lbrace \pi^2 + a \left(2(1-\cos 2\pi c)-2\pi^2(1+c+2c^2)\right) + a^2 A_2(a_0,c)\right\rbrace
$$
with
$$
A_2(a_0,c)= A_1(a_0,c)+\pi^2(3+4c+9c^2+8c^3)-4(1+c)(1-\cos 2\pi c).
$$
Now we need to estimate from below the denominator $D(a,c)$. For that purpose, we need the following result.
\begin{lemma}\label{lemmaG}
Let $G(a,c)$ defined by
$$G(a,c)=\frac{2c-a}{2(c-a)} (\cos 2\pi a-\cos 2\pi c) - \frac{2(c-a)}{2c-a} (\cos \pi a - \cos 2\pi c) .$$
Then
\begin{equation}\label{boundlemma}
G(a,c) \leq \frac{1-\cos 2\pi c}{c} \,a .
\end{equation}
\end{lemma}
\begin{proof}
Let us denote $G_1(a,c)=G(a,c)-\frac{1-\cos 2\pi c}{c} \,a$. We want to prove that $G_1(a,c)\leq 0$.
We set $c=a+x$ and we use the inequality (coming from a simple Taylor expansion)
\begin{equation}\label{lem1a}
\cos(2\pi a+2\pi x)\geq \cos 2\pi a-2\pi x\sin 2\pi a-2\pi^2 x^2 \cos 2\pi a.
\end{equation}
In the sequel, we introduce $C=\cos 2\pi a$ and $S=\sin 2\pi a$. Expanding $G_1(a,c)$ and using \eqref{lem1a}, we get
$$
\begin{array}{l}
G_1(a,c)\leq (a+2x)(\pi S+\pi^2 xC)-\frac{2x}{a+2x}\left(\cos \pi a-C+2\pi xS+2\pi^2 x^2 C\right) \ldots \\
-\frac{a}{a+x} \left(1-C+2\pi xS+2\pi^2 x^2 C\right).
\end{array}$$
Reducing at the same denominator finally provides (the terms in $x^4$ and $x^3$ cancel) 
$$G_1(a,c)\leq \frac{1}{(a+x)(a+2x)}\left(B_2x^2+B_1 x+B_0\right)$$
where $B_0,B_1,B_2$ are given by
$$B_0=\pi a \sin 2\pi a -(1-\cos 2\pi a),\quad B_2=3\pi^2 a^2 \cos 2\pi a-2(\cos \pi a-\cos 2\pi a)$$
and
$$B_1=\pi^2 a^3 \cos 2\pi a +3\pi a^2\sin 2\pi a -2a(1+\cos \pi a-2\cos 2\pi a).$$
It is immediate to check that $B_0$ can be written as $B_0=(\sin 2\pi a)(\pi a-\tan \pi a)$, thus it is negative for $0<a<1/2$.\\
For $B_1$ and $B_2$ we just use the classical inequalities
$$\sin u\leq u-\frac{u^3}{6}+\frac{u^5}{120},\quad 1-\frac{u^2}{2} \leq  \cos u\leq 1-\frac{u^2}{2} +\frac{u^4}{24}$$
that give
$$B_2\leq \pi^4 a^4 \left(2\pi^2 a^2-\frac{14}{3}\right)$$
showing that $B_2\leq 0$ when $a\in [0,1/4]$. In the same way
$$B_1\leq \pi^4 a^5 \left(\frac{22}{15} \pi^2 a^2 -\frac{10}{3}\right)$$
and $B_1\leq 0$ for $a\in [0,1/4]$. This proves the lemma.
\end{proof}
Using Lemma \ref{lemmaG} and standard estimates such as $\sin 2\pi a \geq a (\sin 2\pi a_0)/a_0$, 
we obtain
$$I_2(a,c) \leq \frac{\sqrt{3}}{8\pi^2}\left\lbrace a\left(\frac{1-\cos 2\pi c}{c} +\pi^2 c\right) -a^2\left(\pi^2+\frac{\pi}{a_0} \sin 2\pi a_0\right) \right\rbrace.$$
It follows:
\begin{equation}\label{boundDac}
16\pi^2 D(a,c) \geq 6\sqrt{3}\pi^2\left(1-a(5c+\frac{2}{\pi^2}\frac{1-\cos 2\pi c}{c})+a^2(\frac{2\sin 2\pi a_0}{\pi a_0} -4)\right).
\end{equation}
Therefore we finally get $F(a,c)\leq 6\sqrt{3}(f_1(c) a +f_2(a_0,c) a^2)$ with
\begin{equation}\label{f1}
f_1(c)=2\left(1+\frac{1}{c}\right)(1-\cos 2\pi c) -\pi^2(4c^2-3c+2)\,,
\end{equation}
and
\begin{equation}\label{f2}
\begin{array}{l}
f_2(a_0,c)=-2\pi \frac{\sin 2\pi a_0}{a_0} +(1-\cos 2\pi c)\left[\frac{3-4a_0}{(1-a_0)(1-2a_0)} -4(1+c)\right] \ldots \\
+\pi^2\left[5+12c+9c^2+8c^3 +2\max\left((1-c)^2,\frac{(1-a_0-c)^2}{(1-a_0)(1-2a_0)}\right)\right].
\end{array}
\end{equation}
It is immediate to check that $f_1(c)<0$ for all $c\in [0,1/2]$ (by estimating $1-\cos u$ by $u^2/2 - u^4/24 + u^6/720$).
Let us now fix $a_0=1/60$. If we can prove that 
$$a\leq \frac{1}{60} \leq -\frac{f_1(c)}{f_2(\frac{1}{60},c)} \quad \Longleftrightarrow f_1(c) + \frac{1}{60}\,f_2(\frac{1}{60},c) \leq 0$$
for any $c\in [0,1/2]$ we will prove the first part of Lemma \ref{smalla}.\\
When $a_0=1/60$, the maximum $M$ appearing in $f_2$ is given by
$$M(\frac{1}{60},c)=\left\lbrace\begin{array}{lc}
(1-c)^2 & \mbox{when } c\geq c_0\\
\frac{(59-60c)^2}{3422} & \mbox{when } c\leq c_0\\
\end{array}\right.$$
where $c_0=(118 -  \sqrt{3422})/178 \simeq 0.3342$. 

Now we use different polynomial estimates for $1-\cos 2\pi c$, according to $c\leq c_0$ or $c\geq c_0$. These estimates allow us to
bound $ f_1(c) + \frac{1}{60}\,f_2(\frac{1}{60},c)$ by an explicit polynomial expression (of degree 5 or 6) for which it is possible to conclude.
More precisely\begin{itemize}
\item for small values of $c$ ($c\leq c_0$), we use the bound 
$$1-\cos u \leq \frac{u^2}{2} - 0.03 u^4 \quad \mbox{valid for $u\in [0,\pi]$}$$
that provides the inequality
$$ f_1(c) + \frac{1}{60}\,f_2(\frac{1}{60},c) \leq \pi^2P_5(c)$$
where $P_5(c)\simeq 0.31582 c^5 -9.40264 c^4  -9.47482 c^3 + 0.15459 c^2 + 7.131 c -1.9493$
and we can check that $P_5(c) < 0$ when $c\in [0,c_0]$.
\item when $u\in [2\pi/3,\pi]$ (this will work for $u=2\pi c$ and $c\in [c_0,1/2]$) we use the following estimate of $1-\cos u$
inspired by the Remes algorithm that provides the best uniform approximation of a function:
$$1-\cos u \leq \alpha_0 +\alpha_1 u+\alpha_2 u^2 +\alpha_3 u^3 +\alpha_4 u^4$$
with 
$$\left\lbrace
\begin{array}{l}
\alpha_0=0.7280333 \\
\alpha_1=-1.46638984\\
\alpha_2=1.67079499\\
\alpha_3=-0.45379449\\
\alpha_4=0.03551166\\
\end{array}\right.$$
It follows that 
$$ f_1(c) + \frac{1}{60}\,f_2(\frac{1}{60},c) \leq \frac{\pi^2}{c}P_6(c)$$
with $P_6(c)\simeq -0.37385 c^6 + 11.89046 c^5 -11.7331 c^4 -13.30008 c^3  + 14.64193 c^2 -3.70387 c +0.14753$
and we can check that $P_6(c) < 0$ when $c\in [c_0,1/2]$.
\end{itemize}
Therefore, we have proved that $F(a,c)\leq 0$ for all $a\leq 1/60$ and for all $c\in [0,1/2]$.

\medskip
Now, let us consider $a_0=0.16=4/25$ and use the inequality $1-\cos 2\pi c \leq 2\pi^2 c^2$. The quantity $f_1+\frac{4}{25} f_2$ is estimated by
$$f_1(c)+\frac{4}{25} f_2(\frac{4}{25},c) \leq \frac{\pi^2}{425} \left\lbrace \frac{18228}{21 } c^2 +3391  c - 342 - \frac{850}{\pi} \sin\left(\frac{8\pi}{25}\right) \right\rbrace.$$
Then it follows immediately that $f_1+\frac{4}{25} f_2 \leq 0$ as soon as $c\in [0,0.16]$ yielding the last part of Lemma \ref{smalla}.
\end{proof}

\section{Proof of Lemma \ref{estimategrad}}

\begin{proof}
In this proof we will give an estimate of the partial derivatives of the function $F$ introduced in \eqref{defF}, namely
$$
F(a,c)=\big(\widetilde{P}(a,c)\big)^2 N(a,c) - 16\pi^2 D(a,c),
$$
in the three regimes $\mathcal Z_I$, $\mathcal Z_{II}$, and $\mathcal Z_{III}$ introduced in \eqref{defZi}, corresponding to
\begin{itemize}
\item[-] region $\mathcal Z_{I}$: $k=0$, $a\in [1/60, 0.06]$, $c\in [0.16, 0.5]$;
\item[-] region $\mathcal Z_{II}$: $k=0.06$, $a\in [0.06,0.12]$, $c\in [0.16, 0.5]$;
\item[-] region $\mathcal Z_{III}$: $k=0.12$, $a\in [0.12, 0.25]$, $c\in [\max(a,0.16), 0.5]$.
\end{itemize}
For brevity, in the following we will denote by $a_{min}$, $a_{max}$, $c_{min}$, and $c_{max}$ the minimal and maximal values of $a$ and $c$ in the considered range.

We start the analysis from the perimeter term $\widetilde{P}$, which does not depend on $k$. By set inclusion, we know that
$$
\widetilde{P}(a,c) \leq P(\Omega_a)=3 (1-a)  \leq 3 (1-a_{min}).
$$
Similarly, denoting by $R$ the decreasing function $R(a):=\sqrt{1-3a+3a^2}$, we have
$$
\widetilde{P}(a,c)  \geq P(H_a)=3R(a)  \geq 3 R(a_{max}).
$$
Using the expression \eqref{ptildeexp} of $\widetilde{P}(a,c)$, we infer that its partial derivatives are
\begin{align*}
\frac{\partial \widetilde P}{\partial a} &
= - \frac{3\big( 2R(a) - 1 \big)}{R(a) (1-2a)^2 } \left[ R(a) (1-2a) +   c-a \right],
\\
\frac{\partial \widetilde P}{\partial c} &= - \frac{6}{1-2a}\left[ 1-a -R(a)\right].
\end{align*}
It is easy to prove that, for $a$ and $c$ in the considered range, the former derivative is negative, decreasing in $c$ and increasing in $a$, whereas the latter is negative, decreasing in $a$ (and independent of $c$). Therefore
\begin{align*}
&\frac{\partial \widetilde P}{\partial a}(a,c)  \leq \frac{\partial \widetilde P}{\partial a} (a_{max}, c_{min}), \quad \frac{\partial \widetilde P}{\partial a}(a,c)  \geq \frac{\partial \widetilde P}{\partial a} (a_{min}, c_{max}),
\\
&\frac{\partial \widetilde P}{\partial c}(a,c)  \leq \frac{\partial \widetilde P}{\partial c} (a_{min}), \ \quad  \quad \quad  \frac{\partial \widetilde P}{\partial c}(a,c)  \geq \frac{\partial \widetilde P}{\partial c} (a_{max}).
\end{align*}
Let us now consider the numerator $N$. First we notice that in view of \eqref{trick2}, we may write
$$
\int_{H_a} |\nabla v_k|^2 \, \mathrm{d}x\mathrm{d}y = \int_{\Omega_a} |\nabla v_k|^2  \, \mathrm{d}x\mathrm{d}y- 2 \int_{ABC}V \, \mathrm{d}x\mathrm{d}y.
$$
This allows us to split $N$ as follows:
$$
N(a,c)=\underbrace{\int_{\Omega_a} |\nabla v_k|^2 \, \mathrm{d}x\mathrm{d}y}_{\hat{N}(a)}- 2 \underbrace{\int_{Q_1(c)BQ_2(c)} V(x,y)\, \mathrm{d}x\mathrm{d}y}_{I_N(a,c)}.
$$
We start with the analysis of $\hat{N}$: it is clearly positive and, by the inclusion $\Omega_{1/2}\subset \Omega_a \subset \Omega_0$, it satisfies
\begin{equation}\label{Nhatbdd}
\hat{N}(1/2) \leq \hat{N}(a) \leq \hat{N}(0).
\end{equation}
Recalling that $v_k=(1-k)u_1 + k \hat{u}_1$ and using formulas in Lemma \ref{lem-F12} and in Remark \ref{remT}, we obtain
$$
\hat{N}(a)= (1-k)^2 \left[ \frac{2}{\sqrt{3}}\pi^2 - F_2(a)  \right] + k^2 \left[\frac{8}{\sqrt{3}}\pi^2 - F_4(a)\right] + 2k (1-k) \left[ - F_6(a)\right].
$$
By combining this formula with the inequalities \eqref{Nhatbdd} we obtain
\begin{align*}
\hat{N}(a) &\leq \frac{2\pi^2}{\sqrt{3}} (5k^2-2k+1),
\\
\hat{N}(a) &\geq  \frac{1}{2\sqrt{3}} \left[ k^2 (5 \pi^2 - 28) + k (-2\pi^2 + 24) + \pi^2 +4\right].
\end{align*}
The function $\hat{N}$ does not depend on $c$, thus we only have to compute the derivative with respect to $a$, denoted for brevity by $N'$:
$$
\hat{N}'(a)= - \left[(1-k)^2 F_2'(a) + k^2 F_4'(a) + 2k (1-k) F_6'(a)\right].
$$
In the case $\mathcal Z_{I}$, namely when $k=0$, it is immediate to check that $\hat{N}'$ is negative decreasing, since
\begin{align*}
F_2'(a) = &  \frac{4 \pi \sqrt{3}}{3} [\pi a (1-\cos(2 \pi a )) + (2\pi a - \sin(2\pi a))] \geq 0, 
\\
 F_2''(a) =  & \frac{4 \pi^2 \sqrt{3}}{3} [ 2  \pi a \sin (2 \pi a ) + 3 (1 - \cos(2\pi a))]\geq 0.
\end{align*}
The same holds true for $k\neq 0$, but the proof is more delicate. We first notice that since
\begin{align*}
F_4'(a) &= \frac{4 \pi \sqrt{3}}{3} [4\pi a(1-\cos(4\pi a)) + 2(4\pi a - \sin(4 \pi a))], 
\\
F_6'(a) &= - \frac{4\pi \sqrt{3}}{3}[3\sin(2\pi a) (1-\cos(2\pi a)) + \cos(2\pi a) (2 \pi a - \sin(2 \pi a))],
\end{align*}
thus $\hat{N}'(0)=0$. Computing the second derivatives of $F_4$ and $F_6$
\begin{align*}
F_4''(a) &= \frac{4 \pi^2 \sqrt{3}}{3} [12 (1-\cos(4\pi a)) + 16 \pi a \sin(4 \pi a) ]
\\
F_6''(a) &= - \frac{4\pi^2 \sqrt{3}}{3}[
8 \cos(2 \pi a) (1-\cos(2 \pi a)) - 4 \sin(2 \pi a) ( \pi a - 2 \sin (2 \pi a))],
\end{align*}
we get
\begin{align*}
\hat{N}''(a)=& -[(1-k)^2 F_2''(a) + k^2F_4''(a) + 2k(1-k)F_6''(a)]
\\
= &   - \frac{4 \pi^2 \sqrt{3}}{3} \left\{
 2 (1-k)(1+3k)\pi a \sin (2 \pi a ) \right.
\\
& + (1-k) \big(3-  3 k  - 16 k \cos(2 \pi a) \big) \big(1 - \cos(2\pi a)\big)
\\
&  \left.   - 8 k (2-5 k) \sin^2(2\pi a)  +  16 k^2  \pi a \sin(4 \pi a) \right\}.
\end{align*}
Using 
$$
3-  3 k  - 16 k \cos(2 \pi a) \geq  3-19 k, \quad \sin(4\pi a) \geq 0,
$$
we get
\begin{align*}
\hat{N}''(a) \leq &   - \frac{4 \pi^2 \sqrt{3}}{3} \left\{
  2 (1-k)(1+3k)\pi a \sin (2 \pi a ) + (1-k)(3-19 k) \big(1 - \cos(2\pi a)\big) \right.
\\
&  \left.  - 8k(2-5k) \sin^2(2\pi a)  \right\}.
\end{align*}
Splitting the coefficient of $\sin^2(2\pi a)$ as the sum of $7k(2-5k)$ and $k (2-5k)$ and noticing that, for $k=0.06, 0.12$, there hold
$$
2 (1-k)(1+3k)\pi a - 7 k(2-5k) \sin(2\pi a)  \geq   2 \pi a [ (1-k)(1+3k) - 7 k (2-5k)  ] \geq  0
$$
and
\begin{align*}
 & (1-k)(3-19 k)(1-\cos(2 \pi a)) - k(2-5k)\sin^2(2 \pi a) 
\\ & = (1-\cos( 2 \pi a)) [(1-k)(3-19 k) - k(2-5k)(1+\cos(2 \pi a))]
\\ & \geq (1-\cos( 2 \pi a)) [ (1-k)(3-19 k) -2 k(2-5k)]\geq 0,
\end{align*}
we deduce that $\hat{N}''(a) \leq 0$. All in all, we conclude that $\hat{N}'$ is decreasing and, since it is zero at $0$, it is negative. Thus we get the estimates
$$
\hat{N}'(a_{max}) \leq \hat{N}'(a) \leq \hat{N}'(a_{min}).
$$
The same strategy applies to the denominator: we write
$$
D(a,c)=\underbrace{\int_{\Omega_a} v_k^2\, \mathrm{d}x\mathrm{d}y}_{\hat{D}(a)} - 2 \underbrace{\int_{ABQ_2(c)}U(x,y)\,\mathrm{d}x \mathrm{d}y}_{I_D(a,c)}.
$$
Recalling that $v_k= (1-k)u_1 + k \hat{u}_1$ and using the formulas in Lemma \ref{lem-F12} and Remark \ref{remT}, we obtain
$$
\hat{D}(a)= (1-k)^2 \left[ \frac{3\sqrt{3}}{8} -F_1(a)  \right]  + k^2 \left[ \frac{3\sqrt{3}}{8} - F_3(a)\right] + 2k(1-k) \left[-F_5(a) \right].
$$
Therefore
$$
\hat{D}'(a)= - \left[(1-k)^2 F_1'(a) + k^2 F_3'(a) + 2k (1-k) F_5'(a)\right].
$$
We claim that $\hat{D}'(a)$ is decreasing, so that
$$
\hat{D}'(a_{max}) \leq \hat{D}'(a) \leq \hat{D}'(a_{min}).
$$
To prove the claim, we first show that $\hat{D}'(0)=0$ and then we prove that $\hat{D}''\leq 0$. The first property comes from
\begin{align*}
F_1'(a)=& \frac{3\sqrt{3}}{8 \pi}[ 6 \pi  a  + 4 \sin(2 \pi a)  + 4 \pi a \cos(2 \pi a)] ,
\\
F_3'(a)= &  \frac{3\sqrt{3}}{8 \pi}[6 \pi a +  2 \sin(4 \pi a)  + 4 \pi a \cos( 4 \pi a)] ,
\\
F_5'(a)= &\frac{3\sqrt{3}}{8 \pi}[ -  3 \sin(2 \pi a) - 2 \sin (4 \pi a )  - 4 \pi a \cos( 2 \pi a) ].
\end{align*}
Differentiating $\hat{D}'$ and using
\begin{align*}
F_1''(a)=& \frac{3\sqrt{3}}{4}[ 3   + 6  \cos(2 \pi a)  - 4 \pi a \sin(2 \pi a)] ,
\\
F_3''(a)= &  \frac{3\sqrt{3}}{4}[3 + 6  \cos(4 \pi a)  - 8  \pi a \sin( 4 \pi a)] ,
\\
F_5''(a)=& \frac{3\sqrt{3}}{4}[ -  5 \cos(2 \pi a) - 4   \cos (4 \pi a )  + 4 \pi a \sin( 2 \pi a) ],
\end{align*}
we obtain
\begin{align*}
\hat{D}''(a)=& - \left[(1-k)^2 F_1''(a) + k^2 F_3''(a) + 2k (1-k) F_5''(a)\right]
\\ =
& - \frac{3\sqrt{3}}{4} \left\{ (1-k)[ 3 - 11 k - (2-2k + 0.5(1-5k)) \pi a \sin(2 \pi a)
]  \right.
\\ &
+ k^2 [  9 - 8  \pi a \sin( 4 \pi a)] + (1-k) (6 - 16 k)  \cos(2 \pi a)
\\
&\left.
+ \sin(2\pi a) [ 4 k (4 - 7k) \sin(2\pi a)   -  \pi a  1.5 (1-k) (1- 5 k) ]
\right\}\leq 0.
\end{align*}
To get the second inequality we have simply rearranged in a convenient way the terms. The last inequality is immediate for $k=0$ (actually in that case $D''= - F_1''\leq 0$ for $a\in [1/60,0.06]$), whereas for $k=0.06$ and $k=0.12$ it follows from the positivity of each term in the expression between braces, for $a$ taken in the proper interval:
\begin{align*}
\bullet &\ 3 - 11 k - (2.5 - 4.5 k)  \pi a \sin(2 \pi a) \geq  3 - 11 k -  (2.5 - 4.5 k) \pi a >0;
\\ 
\bullet &\  9  - 8  \pi a \sin( 4 \pi a) \geq  9 - 8 \pi a >0;
\\
\bullet & \ (1-k)(6 - 16k)  \cos(2 \pi a)>0;
\\
\bullet &\  4 k (4 - 7k) \sin(2\pi a) - \pi a  1.5 (1-k) (1- 5 k) \geq 0.
\end{align*}
This concludes the claim about $\hat{D}'$. 
We conclude the estimates by treating the integral terms $I_N$ and $I_D$. Let us recall their definition:
\begin{eqnarray}
I_N(a,c)= &  \int_{T_N} V(x,y)\, \mathrm{d}x \mathrm{d}y,
\\
I_D(a,c)= &  \int_{T_D} U(x,y) \, \mathrm{d}x \mathrm{d}y,
\end{eqnarray}
where, for brevity,  $T_N$ denotes the triangle $Q_1(c)BQ_2(c)$ and $T_D$ the triangle $ABQ_2(c)$.
For the benefit of the reader, we also recall here the expressions of the integrands:
\begin{eqnarray}
V(x,y)= & \frac{8}{3}\pi^2[(1-k)^2(3-\varphi) + 4 k^2 (3-\eta) + 4 k (1-k) (\psi - \varphi)] , \label{defN}
\\
U(x,y)= &  \frac{3}{2} [(1-k)^2(3+2\varphi) + k^2 (3+2\eta) + 2 k (1-k) (-\psi - 2\varphi)] . \label{defD}
\end{eqnarray}
The dependence on $a$ and $c$ of $I_N$ and $I_D$ is encoded in the domains of integration:
\begin{eqnarray}
T_N& =\left\{ y_1(c) \leq y \leq y_2(a), \ x_N(a,c,y) \leq x \leq x_2(y) \right\},\label{TN}
\\
T_D& =\left\{ y_1(c) \leq y \leq y_2(a), \ x_D(a,c,y) \leq x \leq x_2(y) \right\},\label{TD}
\end{eqnarray}
with
\begin{align*}
& y_1(c):=  \frac{\sqrt{3}}{2}(1-c), \quad y_2(a):=\frac{\sqrt{3}}{2}(1-a), \quad x_2(y):= - \frac{y}{\sqrt{3}} + \frac12,
\\
& x_N(a,c,y):=\frac{1}{1-2a} \left(-\frac{y}{\sqrt{3}} +  \frac12 - ac\right), \quad x_D(a,c,y):=\frac{c}{c-a} \left(-\frac{y}{\sqrt{3}} +  \frac12 - \frac{a}{2}\right).
\end{align*}
We start by bounding $I_N$: using the positivity of the integrand, we get $0 \leq I_N(a,c) \leq \|V\|_\infty |T_N|$. On the one hand, since $|\varphi|, \, |\psi|,\, |\eta| \leq 3$, we easily get $\|V\|_\infty \leq  16\pi^2 (1+ k)^2$. On the other hand, the area of $T_N$ is bounded above by the area of the triangle $ABC$, which is maximal for $a=a_{max}$. All in all, we obtain
$$
0 \leq I_N \leq  \sqrt{3}\pi^2 (1+k)^2  a_{max} (1-2a_{max}).
$$
In order to compute the partial derivatives of $I_N$, we write the integral as follows:
$$
I_N(a,c)= \int_{y_1(c)}^{y_2(a)} \int_{x_N(a,c,y)}^{x_2(y)} V (x,y)\, \mathrm{d}x\mathrm{d}y.
$$
Differentiating in $a$ and $c$, we obtain
$$
\frac{\partial I_N}{\partial a}(a,c)= H_1 + H_2, \quad\frac{\partial I_N}{\partial c}(a,c)= H_3,
$$
with
\begin{align*}
&H_1:={y_2'(a)} \int_{x_N(a,c,y_2(a) )}^{x_2(y_2(a) )} V(x,y_2(a))\, \mathrm{d}x,
\\
& H_2:=-  \int_{y_1(c)}^{y_2(a)}  V (x_N(a,c,y),y) \frac{\partial x_N}{\partial a}(a,c,y)\,\mathrm{d}y,
\\
& H_3:= - \int_{y_1(c)}^{y_2(a)} V (x_N(a,c,y),y) \frac{\partial x_N}{\partial c}(a,c,y)\,\mathrm{d}y.
\end{align*}
For the term $H_3$ we have used $x_2(a,c,y_1(c))=x_N(a,c,y_1(c))$.

We notice that $H_1$ is the product of a negative coefficient, $y_2'(a)$, and an integral of a positive function, $V$, over a segment of length $x_2(y_2(a))-x_N(a,c,y_2(a))=a (c-a)/(1-2a)$. Bounding $\|V\|_\infty$ as above and optimizing in $c$ and $a$, we get
$$
-  4 \sqrt{3} \pi^2 (1+k)^2 a_{max} \leq H_1 \leq 0.
$$
Since $\partial x_N/\partial a \leq 0 $  for $y\in [y_1(c), y_2(a)]$ and $V\geq 0$, we infer that $H_2\geq 0$. 
In order to obtain an upper bound of $H_2$, we perform the change of variable
\begin{equation}\label{chvar}
y=y(t) := \frac{\sqrt{3}}{2}(1-c) + \frac{\sqrt{3}}{2} t, 
\end{equation}
which allows to re-parametrize the interval of integration $[y_1(c), y_2(a)]$ as the interval $[0, c-a]$. Notice that, as a function of $t$, $|\partial x_N /\partial a|$ simply reads $t/(1-2a)^2$. Thus we obtain
$$
\int_{y_1(c)}^{y_2(a)} \left| \frac{\partial x_N}{\partial a}(a,c,y)  \right|\, \mathrm{d}y = 
 \frac{\sqrt{3}}{4} \frac{(c-a)^2}{(1-2a)^2},
$$
which allows us to obtain
$$
0 \leq H_2 \leq 4 \sqrt{3} \pi^2(1+k)^2 \frac{(c_{max}-a_{min})^2}{(1-2a_{min})^2}.
$$
Similarly, by direct computation we have
$$
\int_{y_1(c)}^{y_2(a)} \left| \frac{\partial x_N}{\partial c}(a,c,y)  \right|\, \mathrm{d}y = \frac{\sqrt{3}}{2} \frac{a (c-a)}{1-2a},
$$
so that
$$
0 \leq H_3 \leq 
8 \sqrt{3} \pi^2 (1+k)^2 \frac{a(c-a)}{1-2a} \leq 
4 \sqrt{3} \pi^2 (1+k)^2  a_{max}.
$$
We now pass to the study of $I_D$. As already done for $I_N$, we write the integral as
\begin{align*}
I_D(a,c)= & \int_{y_1(c)}^{y_2(a)} \int_{x_D(a,c,y)}^{x_2(y)} U (x,y)\, \mathrm{d}x\mathrm{d}y.
\end{align*}
In this form, it is easy to compute the partial derivatives, which read
$$
\frac{\partial I_D}{\partial a}(a,c)= J_1 + J_2, \quad \frac{\partial I_D}{\partial c}(a,c)= J_3, 
$$
with
\begin{align*}
& J_1 := {y_2'(a)} \int_{x_D(a,c,y_2(a) )}^{x_2(y_2(a) )} U (x,y_2(a))\, \mathrm{d}x,
\\
& J_2 :=  - 
 \int_{y_1(c)}^{y_2(a)}  U (x_D(a,c,y),y) \frac{\partial x_D}{\partial a}(a,c,y)\, \mathrm{d}y,
\\
& J_3 :=  - 
 \int_{y_1(c)}^{y_2(a)} U (x_D(a,c,y),y) \frac{\partial x_D}{\partial c}(a,c,y)\, \mathrm{d}y.
\end{align*}
For $J_3$ we have used the fact that $x_2(y_1(c))=x_D(a,c,y_1(c))$. By combining the sign of the integrands
$$
U\geq 0, \quad y'(a)<0, \quad \frac{\partial x_D}{\partial a}\leq 0, \quad \frac{\partial x_D}{\partial c} \leq 0,
$$
true for $y \in [y_1(c), y_2(a)]$, we immediately obtain $J_1\leq0$, $J_2\geq0$, $J_3\geq0$. Let us refine these estimates, starting from $J_1$: by combining $y'(a)=-\sqrt{3}/2$, the fact that the segment of integration is $AB$ with length $a/2$, and the estimate $U(B) \leq U(x,y) \leq U(A)$ provided in Lemma \ref{lemmas}, we deduce that 
$$
 - \frac{\sqrt{3}}{4} a U(A)  \leq J_1 \leq  - \frac{\sqrt{3}}{4} a U(B). 
$$
Inserting the coordinates of $A$ and $B$ (see \eqref{ABC}) into the expression of $U$ (see \eqref{defD} and \eqref{fipsieta}), we obtain
\begin{align*}
U(A)=& \frac{3}{2}\{(1-k)^2 [3+2  \cos(2 \pi a)  + 4 \cos( \pi a) ] + 
k^2 [3+2 \cos(4\pi a) + 4 \cos(2 \pi a)]
\\
&
+ 2 k (1-k) [ -1 - 2\cos( 3 \pi a) - 2\cos(2 \pi a)  - 4 \cos(\pi a) ]\},
\\ U(B)=& \frac{3}{2} \{ (1-k)^2 [5+ 4\cos(2\pi a)] + k^2 [5+4\cos(4\pi a)] 
\\ & + 2k(1-k) [- 6 \cos(2\pi a) - \cos(4\pi a) - 2]
  \}.
\end{align*}
A direct computation shows that $a\mapsto aU(A)$ and $a\mapsto aU(B)$ are increasing, thus 
$$
  - \frac{\sqrt{3}}{4} a_{max} U(A) \leq J_1  \leq   - \frac{\sqrt{3}}{4} a_{min} U(B).
$$
For the terms $J_2$ and $J_3$, we argue as for $H_2$ and $H_3$. By performing the change of variables $y=y(t)$ described in \eqref{chvar}, the interval of integration becomes $[0, c-a]$, the partial derivatives in the integrand read
$$
\frac{\partial x_D}{\partial a} (a,c,y(t)) =   - \frac{c}{2(c-a)^2} t, \quad \frac{\partial x_D}{\partial c} (a,c,y(t))= - \frac{ a }{2(c-a)^2}( c-a-t),
$$
and $\mathrm{d}y= (\sqrt{3}/2) \mathrm{d} t$. This allows to write
\begin{align*}
\int_{y_1(c)}^{y_2(a)} \left| \frac{\partial x_D}{\partial a}(a,c,y)  \right|\, \mathrm{d}y = \frac{\sqrt{3}}{8} c,
\quad
 \int_{y_1(c)}^{y_2(a)}  \left| \frac{\partial x_D}{\partial c}(a,c,y)  \right| \, \mathrm{d}y= \frac{\sqrt{3}}{8} a. 
\end{align*}
Thus we get
$$
0 \leq J_2 \leq \frac{27 \sqrt{3}}{16}  c_{max},
\quad 
 0  \leq J_3 \leq \frac{27 \sqrt{3}}{16} a_{max}. 
$$
According to the definitions given in the proof, $F$ reads
$$
F(a,c)= \tilde{P}^2 (\hat{N} - 2 I_N) - 16\pi^2 (\hat{D} - 2 I_D)
$$
and its partial derivatives are
\begin{align*}
 \frac{\partial F}{\partial a} (a,c) = & [- 4 I_N \tilde{P}\partial_a \tilde{P} - 2 \tilde{P}^2H_1 - 16 \pi^2 \hat{D}' + 32 \pi^2 J_2] 
 \\ & + [2 \hat{N} \tilde{P} \partial_a \tilde{P} + \tilde{P}^2 \hat{N}' - 2\tilde{P}^2 H_2 + 32 \pi^2 J_1],
\\ 
 \frac{\partial F}{\partial c} (a,c) = &[- 4 I_N \tilde{P} \partial_c\tilde{P} +  32\pi^2 J_3] + [2 N \tilde{P} \partial_c \tilde{P} - 2 \tilde{P}^2 H_3],
\end{align*}
where we have separated, in square brackets, the positive and negative terms. By taking the proper values of $k$, $a_{min}$, $a_{max}$, $c_{min}$, $c_{max}$ in the three regions $\mathcal Z_{I}$, $\mathcal Z_{II}$, $\mathcal Z_{III}$, and finally passing to the absolute values, we find the desired bounds of the statement. This concludes the proof.
\end{proof}

\begin{remark} \label{remInt}
Let us write the exact formulas of the integrals of $\varphi, \psi , \eta$ defined in \eqref{fipsieta} in the triangles $T_N$ and $T_D$ defined in \eqref{TN} and \eqref{TD}, respectively. These expressions are used in the numerical computation of $F$.
\begin{align*}
\int_{T_N} \varphi= \frac{\sqrt{3}}{8\pi^2}& \left\{ -  \frac{c-a}{1-2a} 2 \pi a \sin(2\pi a)  
+\frac{1-2a}{1-a} \left[\cos(2\pi c) -  \cos\left( 2 \pi a \frac{1-a-c}{1-2a}\right) \right] 
\right.
\\ & +\left. \frac{1-a}{1-2a}\left[ \cos(2\pi a) - \cos (2 \pi c)\right] + \frac{1-2a}{a} \left[1 - \cos\left( 2 \pi a \frac{a-c}{1-2a}\right) \right]
\right\},
\\
\int_{T_N} \psi=\frac{\sqrt{3}}{ 8 \pi^2}& \left\{
\frac{1-2a}{1-3a}
 \left[\cos\left(2 \pi a \frac{1 - 3 a  +  c}{1- 2a}\right) 
- \cos(2\pi c)\right]\right.
\\ &
 + \frac{1-2a}{2-3a}
 \left[ \cos(4\pi c)- \cos \left(2 \pi a \frac{ 2 - 3 a  -  c}{1-2a}\right)\right]
\\&
- 
\left[
\frac{\cos(2 \pi a  )  - \cos(2 \pi c) }{2}  + \frac{\cos(4 \pi c ) - \cos(4 \pi a)}{2}
\right]
\\&
\left.+\frac{(1-2a)}{2}
\left[
\cos\left(2 \pi c \right) - \cos\left(2 \pi a \frac{1 - 2 c}{1-2a}\right)
\right]\right\},
\\
\int_{T_N} \eta=  \frac{\sqrt{3}}{32\pi^2} & \left\{
\frac{1-2a}{a(1-a)}
\left[1 - \cos \left(   4 \pi a  \frac{c-a}{1-2a}\right)\right]  \right.
\\ &
+\frac{1-2a}{ 1-a}\left[
- 1+  \cos(4\pi c)
+ \cos\left(4 \pi a \frac{c-a}{1-2a}\right)-  \cos\left(4 \pi a \frac{1-a-c}{1-2a}\right)
\right]
\\ &
\left.+\frac{1}{1-2a} [(1-a)\cos(4 \pi a ) - (1-a)\cos(4 \pi c) -   4 \pi a(c-a) \sin(4\pi a )  ]\right\},
\end{align*}
and
\begin{align*}
\int_{T_D} \varphi = \frac{\sqrt{3}}{8 \pi^2}&\left\{ -  \pi a \sin(2a\pi)
+ 2 (c-a)\left[ \frac{\cos(2\pi c) - \cos(\pi a)  }{ (2c-a)} + \frac{1 - \cos(\pi a)}{a} \right]\right.
\\ & 
\left. + \frac{(2c-a)}{2 (c-a)}[\cos(2 \pi a) - \cos(2\pi c)]
\right\},
\\
\int_{T_D} \psi= \frac{\sqrt{3}}{8\pi^2} & \left\{ \frac{ \cos(2 \pi  c ) - \cos(2 \pi    a )}{2} + \frac{ \cos(4  \pi  a) - \cos(4 \pi  c)}{2} - (c-a) 
 \frac{1-\cos(2 \pi c)}{2c}\right.
\\ &
\left. + 
\frac{2(c-a)}{3a-4c}
\left[
-  \cos(4\pi c) + \frac{2 c \cos(3  \pi a) 
+( 3a-4c)\cos( 2  \pi c )}{ (3a-2c)}
\right]\right\},
\\
\int_{T_D} \eta= 
 \frac{\sqrt{3}}{64 \pi^2} &\left\{ - 4 a  \pi  \sin(4 \pi  a)
- 4 \frac{c-a}{2c-a}  - \frac{a}{2c-a} \cos(4 \pi a) 
\right.
\\
& \left.
+ \frac{4c-3a}{2c-a}\frac{ c\cos(4 \pi a)  - a\cos(4\pi c) }{c-a}
+\frac{8c(c-a)}{2c-a}\frac{1-\cos(2\pi a)}{a} \right\}.
\end{align*}
\end{remark}

\bibliography{biblio}

\begin{thebibliography}{10}

\bibitem{BB}
Rodrigo Ba\~{n}uelos and Krzysztof Burdzy.
\newblock On the ``hot spots'' conjecture of {J}. {R}auch.
\newblock {\em J. Funct. Anal.}, 164(1):1--33, 1999.

\bibitem{BcB}
Dorin Bucur and Giuseppe Buttazzo.
\newblock {\em Variational methods in shape optimization problems}, volume~65
  of {\em Progress in Nonlinear Differential Equations and their Applications}.
\newblock Birkh\"{a}user Boston, Inc., Boston, MA, 2005.

\bibitem{cheng}
Shiu~Yuen Cheng.
\newblock Eigenvalue comparison theorems and its geometric applications.
\newblock {\em Math. Z.}, 143(3):289--297, 1975.

\bibitem{henrot-ftouhi}
Ilias Ftouhi and Antoine Henrot.
\newblock The diagram $(\lambda_1, \mu_1)$.
\newblock {\em preprint, hal-03311538v1}, 2021.

\bibitem{livre_vert}
Antoine Henrot.
\newblock {\em Extremum problems for eigenvalues of elliptic operators}.
\newblock Frontiers in Mathematics. Birkh\"{a}user Verlag, Basel, 2006.

\bibitem{Henrot-Pierre}
Antoine Henrot and Michel Pierre.
\newblock {\em Shape variation and optimization}, volume~28 of {\em EMS Tracts
  in Mathematics}.
\newblock European Mathematical Society (EMS), Z\"{u}rich, 2018.
\newblock A geometrical analysis, English version of the French publication
  [MR2512810] with additions and updates.

\bibitem{LS}
Richard~S. Laugesen and Bart\l omiej~A. Siudeja.
\newblock Maximizing {N}eumann fundamental tones of triangles.
\newblock {\em J. Math. Phys.}, 50(11):112903, 18, 2009.

\bibitem{LS2}
Richard~S. Laugesen and Bart\l omiej~A. Siudeja.
\newblock Triangles and other special domains.
\newblock In {\em Shape optimization and spectral theory}, pages 149--200. De
  Gruyter Open, Warsaw, 2017.

\bibitem{PW}
Lawrence~E. Payne and Hans~F. Weinberger.
\newblock An optimal {P}oincar\'{e} inequality for convex domains.
\newblock {\em Arch. Rational Mech. Anal.}, 5:286--292 (1960), 1960.

\bibitem{raiko}
Arseny Raiko.
\newblock Comparison of the first positive neumann eigenvalues for rectangles
  and special parallelograms.
\newblock {\em preprint, arXiv 1810.07025}, 2018.

\bibitem{W56}
Hans~F. Weinberger.
\newblock An isoperimetric inequality for the {$N$}-dimensional free membrane
  problem.
\newblock {\em J. Rational Mech. Anal.}, 5:633--636, 1956.

\end{thebibliography}

\bibliographystyle{plain}

\bigskip

\noindent Antoine \textsc{Henrot}, Institut \'Elie Cartan de Lorraine, UMR 7502, Universit\'e de Lorraine CNRS, France, email: \texttt{antoine.henrot@univ-lorraine.fr} 

\medskip

\noindent Antoine \textsc{Lemenant}, Institut \'Elie Cartan de Lorraine, UMR 7502, Universit\'e de Lorraine CNRS, France, email: \texttt{antoine.lemenant@univ-lorraine.fr} 

\medskip

\noindent  Ilaria \textsc{Lucardesi}, Universit\`a degli Studi di Firenze, Dipartimento di Matematica e Informatica ``Ulisse Dini'', Italy, and formerly Institut \'Elie Cartan de Lorraine, UMR 7502, Universit\'e de Lorraine CNRS, France, email: \texttt{ilaria.lucardesi@unifi.it}

\end{document}